\documentclass[preprint,times,letterpaper]{elsarticle}

\usepackage{latexsym}
\usepackage{amsmath,amsthm}
\usepackage{amsfonts}
\usepackage{graphicx,epsfig, epstopdf}
\usepackage[usenames,dvipsnames]{color}
\usepackage{verbatim}

\newcommand{\bitem}{\begin{itemize}}
\newcommand{\eitem}{\end{itemize}}
\newcommand{\beq}{\begin{equation}}
\newcommand{\eeq}{\end{equation}}

\newcommand{\cC}{{\mathcal C}}

\newcommand{\cS}{{\mathcal S}}
\newcommand{\cR}{{\mathcal R}}
\newcommand{\cH}{{\mathcal H}}

\newcommand{\cG}{{\mathcal G}}
\newcommand{\cM}{{\mathcal M}}
\newcommand{\bZ}{{\mathbb Z}}

\newcommand{\bR}{{\mathbb R}}
\newcommand{\bN}{{\mathbb N}}

\newcommand{\bS}{{\mathbb S}}

\newcommand{\Sp}{{\rm{supp }}}
\newcommand{\eSp}{{\rm{esssupp }}}
\newcommand{\sgn}{{\rm{sgn}}}

\newcommand\LpRn[2]{L^{#1}(\bR^{#2})}
\newcommand\Lp[1]{\LpRn{{#1}}{2}}

\newcommand\bm{{\tilde m}}

\newcommand{\ip}[2]{\langle#1,#2\rangle}

\newcommand{\norm}[1]{\|#1\|}

\newtheorem{theorem}{Theorem}[section]
\newtheorem{corollary}{Corollary}[section]
\newtheorem{proposition}{Proposition}[section]
\newtheorem{lemma}{Lemma}[section]
\newtheorem{definition}{Definition}[section]
\newtheorem{example}{Example}[section]

\newcommand{\xz}[1]{{\color{blue}{#1}}}

\newcommand{\bgb}[1]{{\color{OliveGreen}{#1}}}

\begin{document}
\title{Gabor Shearlets}

\author[rvt0]{Bernhard G. Bodmann\fnref{fn0}\corref{cor}}
\ead{bgb@math.uh.edu}
\address[rvt0]{Department of Mathematics, University of Houston, USA}
\fntext[fn0]{Research was supported in part by NSF grant DMS-1109545. Part of this work was completed
while visiting the Technische Universit{\"a}t Berlin and the Erwin Schr{\"o}dinger Institut Wien.}
\cortext[cor]{Corresponding author.}

\author[rvt1]{Gitta Kutyniok\fnref{fn1}}
\ead{kutyniok@math.tu-berlin.de}
\address[rvt1]{Institute of Mathematics, Technische Universit\"at Berlin, Germany}
\fntext[fn1]{Research was supported in part by the Einstein Foundation Berlin, by Deutsche
Forschungsgemeinschaft (DFG) Grant SPP-1324 KU 1446/13 and DFG Grant KU 1446/14, by
the DFG Collaborative Research Center TRR 109 ``Discretization in Geometry and Dynamics'',
and by the DFG Research Center {\sc Matheon} ``Mathematics for key technologies'' in Berlin.
}

\author[rvt2]{Xiaosheng Zhuang}
\ead{xzhuang7@cityu.edu.hk}
\address[rvt2]{Department of Mathematics, City University of Hong Kong, China}

\makeatletter \@addtoreset{equation}{section} \makeatother

\begin{abstract}
In this paper, we introduce Gabor shearlets, a variant of shearlet systems, which are
based on a  different group representation than previous shearlet constructions:
they combine elements from Gabor and wavelet frames in their construction. As a consequence,
they can be implemented with standard filters from wavelet theory in combination with
standard Gabor windows. Unlike the usual shearlets, the new construction can achieve a
redundancy as close to one as desired. Our construction follows the general strategy for
shearlets. First we define group-based Gabor shearlets and then modify them to a
cone-adapted version. In combination with Meyer filters, the cone-adapted Gabor shearlets
constitute a tight frame and provide low-redundancy sparse approximations of the common
model class of anisotropic features which are cartoon-like functions.
\end{abstract}


\maketitle

\section{Introduction}

During the last 10 years, directional representation systems such as curvelets and shearlets
were introduced to accommodate the need for sparse approximations of anisotropic features
in multivariate data. These anisotropic features, such as singularities on lower dimensional
embedded manifolds, called for representation systems to sparsely approximate such
data. Prominent examples in the 2-dimensional setting are edge-like structures in images
in the regime of explicitly given data and shock fronts in transport equations in the regime of implicitly
given data. Because of their isotropic nature, wavelets are not as well adapted to this task
as curvelets \cite{CD04}, contourlets \cite{DV05},
or shearlets \cite{KL12}. Recently, a general framework for directional representation systems based on
parabolic scaling -- a scaling adapted to the fact that the regularity of the singularity in the considered
model is $C^2$ -- was introduced in \cite{GK12} seeking to provide a comprehensive viewpoint towards
sparse approximations of cartoon-like functions.

Each system comes with its own advantages and disadvantages.
%
%
%
Shearlet systems distinguished themselves by the fact that these systems are available as compactly supported
systems -- which is desirable for applications requiring high spatial localization such as PDE solvers --
and also provide a unified treatment of the continuum and digital setting thereby ensuring faithful
implementations.
Shearlets were introduced in \cite{GKL06} with the early theory focussing on band-limited
shearlets, see e.g. \cite{GL07a}. Later, a compactly supported variant was introduced in \cite{KKL10}, which
again provides optimally sparse approximations of cartoon-like functions \cite{KL10}. In contrast to those
properties,  contourlets do not provide
optimally sparse approximations and curvelets are neither compactly supported nor do they treat the
continuum and digital realm uniformly due to the fact that they are based on rotation in contrast to
shearing.


\subsection{Key Problem}
\label{subsec:key}

One major problem -- which might even be considered a ``holy grail'' of the area of geometric multiscale
analysis -- is whether a system can be designed which is
\bitem
\item[(P1)] an orthonormal basis,\\[-4ex]
\item[(P2)] compactly supported,\\[-4ex]
\item[(P3)] possesses a multiresolution structure,\\[-4ex]
\item[(P4)] and provides optimally sparse approximations of cartoon-like functions.
\eitem
Focussing from now on entirely on shearlets, we can observe that bandlimited shearlets satisfy (P4) while
replacing (P1) with being a tight frame. Compactly supported shearlets accommodate (P2) and (P4), and
form a frame with controllable frame bounds as a substitute for (P1). We are still far from being able
to construct a system satisfying all those properties -- also by going beyond shearlets -- , and it is
not even clear whether this is at all possible, cf. also \cite{Hou13}.
Several further attempts were already made in the past. In \cite{KS09}, shearlet systems were introduced
based on a subdivision scheme, which naturally leads to (P2) and (P3), but not (P1) -- not even being
tight -- and (P4). In \cite{HKZ10}, a different multiresolution approach was utilized leading to systems
which satisfy (P2) and (P3), but not (P4), and (P1) only by forming a tight frame without results on
their redundancy.


\subsection{What are Gabor Shearlets?}

The main idea of the present construction is to use a deformation of the group operation that
common shearlet systems are based upon together with a decomposition in the frequency domain to ensure
an almost uniform treatment of different directions, while modeling the systems as closely as possible
after the one-dimensional multiresolution analysis (MRA) wavelets. To be more precise, the
new group operation includes shears and chirp modulations which satisfy the well-studied Weyl-Heisenberg
commutation relations. Thus, the shear part naturally leads us to Gabor frame constructions instead of
an alternative viewpoint in which shears enter in composite dilations \cite{GLLWW06}. The filters
appearing in this construction can be chosen as the trigonometric polynomials belonging to standard
wavelets or to $M$-band versions of them, or as the smooth filters associated with Meyer's construction.
To achieve the optimal approximation rate for cartoon-like functions, we use a cone adaptation procedure.
But in contrast to other constructions, we avoid incorporating redundancy in this step.

It is interesting to notice that due to the different group structure, Gabor shearlets do not fall into the framework of
parabolic molecules (cf. \cite{GK12}) although they are based on parabolic scaling. Thus, this framework
can not be used in our situation for deriving results on sparse approximations by transfering such properties
from other systems.


\subsection{Our Contributions}

Gabor shearlets satisfy the following properties, related to Subsection \ref{subsec:key}:
\bitem
\item[(P1$^*$)] Gabor shearlets can be chosen to be unit norm and $b^{-1}$-tight, where $b^{-1}$ --  which can
be interpreted as the redundancy (cf. Subsection \ref{subsec:redundancy}) -- can be chosen arbitrarily close to one.\\[-4ex]
\item[(P2$^*$)] Gabor shearlets are not compactly supported, but can be constructed with polynomial
decay in the spatial domain.\\[-4ex]
\item[(P3)]The two-scale relation for the shearlet subband decomposition is implemented with standard filters
related to MRA wavelets.\\[-4ex]
\item[(P4)] In conjunction with a cone-adaptation strategy and Meyer filters, Gabor shearlets provide optimally sparse approximations of
cartoon-like functions.
\eitem
Thus, (P3) and (P4) are satisfied. (P1) is approximately satisfied in the sense that the systems with property (P1$^*$) are arbitrarily
close to being orthonormal bases. And (P2) is also approximately satisfied by replacing compact support by
polynomial decay in (P2$^*$). It is in this sense that we believe the development of Gabor shearlets contributes to
introducing a system satisfying (P1)--(P4). Or -- if it could be proven that those are not simultaneously
satisfiable -- providing a close approximation to those.


\subsection{Outline of the Paper}

The remainder of this paper is organized as follows. In Section \ref{sec:Preliminaries}, we set the notation
and recall the essential properties of Gabor systems, wavelets, and shearlets which are needed in the sequel.
In this section, we also briefly introduce the notion of redundancy first advocated in \cite{BCK11}.
In Section \ref{sec:GroupGaborShearlets}, after providing some intuition on our approach we introduce
Gabor shearlets based on a group related to chirp modulations and discuss their frame properties and
the associated multiresolution structure. The projection of those Gabor shearlets on cones in the frequency
domain is then the focus of Section \ref{sec:ConeGaborShearlets}, again starting with the construction
followed by a discussion of similar properties as before. The last section, Section \ref{sec:approx},
contains the analysis of sparse approximation properties of cone-adapted Gabor shearlets.


\section{Revisited: Wavelets, Shearlets, and Gabor Systems}
\label{sec:Preliminaries}

In this section, we introduce the main notation of this paper, state the basic definitions of
Gabor systems, wavelets, and shearlets, and also recall the underlying construction principles,
formulated in such a way that Gabor shearlets will become a relatively straightforward generalization.
We emphasize
that this is not an introduction to Gabor and wavelet theory, and we expect the reader to have some
background knowledge, otherwise we refer  to \cite{Daub:book} or \cite{Mal98}. A good general reference
for most of the material presented in this section is the book by Weiss and Hern\'{a}ndez \cite{WH96}.  In the last
part of this section, we discuss the viewpoint of redundancy from \cite{BCK11}, which we adopt in this paper.

In what follows, the Fourier transform of $f\in \LpRn{1}{n}$ is defined to be $\widehat{f}(\xi):=\int_{\mathbb{R}^n} f(x)e^{-2\pi i
x\cdot\xi}dx$, where $x\cdot \xi$ is the dot product between $x$ and $\xi$ in $\mathbb R^n$. As usual, we extend this integral
transform to the unitary map $f \mapsto \widehat f$ defined for any function $f$ which is square integrable. The unitarity is
captured in the Plancherel identity $\ip{f}{g}=\ip{\widehat f}{\widehat g}$ for any two functions $f, g \in L^2(\mathbb R^n)$ with $\ip{f}{g}:=\int_{\mathbb R^n} f(x)\overline{g(x)}dx$.


\subsection{MRA Wavelets}
\label{subsec:MRAwavelets}

Let $\{\phi,\psi\}$ be a pair of a scaling function and a wavelet for $L^2(\mathbb R)$  associated with a pair of a
low-pass filter $H: \mathbb T \to \mathbb C$ and a high-pass filter $G: \mathbb T \to \mathbb C$, for convenience
defined on the torus $\mathbb T = \{z \in \mathbb C: |z| = 1\}$. We start by recalling the Smith-Barnwell condition
for filters.

\begin{definition}
A filter $H: \mathbb T \to \mathbb C$ satisfies the {\em Smith-Barnwell condition}, if
\[
   |H(z)|^2+|H(-z)|^2 = 1
\]
for almost every $z \in \mathbb T$.
\end{definition}

The Smith-Barnwell condition is an essential ingredient in the characterization of localized multiresolution analyses; that is, the scaling functions $\phi$  are localized in the sense of having faster than polynomial decay: $\int_{\mathbb{R}} (1+x^2)^n|\phi(x)|^2dx<\infty$ for all $n\in\mathbb{N}$.

\begin{theorem}[Cohen, as in \cite{WH96} Theorem 4.23 of Chapter 7]
A $C^\infty$ function $H: \mathbb T \to \mathbb C$ is the low-pass filter of a localized multiresolution analysis
with scaling function $\phi$ given by
$$
  \widehat \phi(\xi) = \prod_{j=1}^\infty  H( e^{-2\pi i \xi/2^j})
$$
if and only if $H(1)=1$, $H$ satisfies the Smith-Barnwell condition, and there exists a set $K\subset \mathbb T$ which
contains $1$ and has a finite complement in $\mathbb T$ such that $H(z^{2^{-j}})\ne 0$ for all $j \in \mathbb Z$, $j \ge 0$
, and $z \in K$.
\end{theorem}

The two-scale relations for $\phi$ and $\psi$ are conveniently expressed in the frequency domain,
\[
\widehat\phi(2\xi) = H(e^{-2\pi i \xi}) \widehat\phi(\xi)\quad\mbox{and}\quad
\widehat\psi(2\xi) =G(e^{-2\pi i \xi}) \widehat\phi(\xi) , \quad \mbox{a.e. } \xi \in \mathbb R.
\]
The orthonormality of the integer translates of $\{\phi, \psi\}$ is captured in the matrix identity
\[
\begin{aligned}
\cM(z)\cM(z)^*=I_2\quad\mbox{with}\quad
\cM(z):=
\left[
\begin{matrix}
H(z) & H(-z)\\
G(z)  & G(-z) \\
\end{matrix}
\right]
\end{aligned} \, ,\quad   \mbox{ for a.e.\ } z \in \mathbb T \, .
\]
Often, only $H$ is specified and the matrix has to be completed to a unitary, with a common choice
being $G(z)=-z\overline{H(-z)}$.

The low-pass filter of the Meyer
scaling function is of particular use for the construction of Gabor shearlets, which will be shown in Section \ref{sec:approx} to yield optimal sparse approximations.
The Meyer scaling function $\phi$ and wavelet function $\psi$ are given by
\[
\widehat\phi(\xi) = \begin{cases}
1 &  \mbox{if }|\xi|\le \frac13,\\
\cos(\frac{\pi}{2}\nu(3|\xi|-1)) & \mbox{if } \frac13\le |\xi|\le \frac23,\\
0 & \mbox{otherwise},
\end{cases}
\]
and
\[
\widehat\psi(\xi)=
\begin{cases}
-e^{-\pi i \xi}\sin\left[\frac{\pi}{2}\nu(3|\xi|-1)\right]& \mbox{if } \frac{1}{3}\le|\xi|\le \frac{2}{3},\\
-e^{-\pi i \xi}\cos\left[\frac{\pi}{2}\nu(\frac{3}{2}|\xi|-1)\right]& \mbox{if } \frac{2}{3}\le|\xi|\le \frac{4}{3},\\
0 & \mbox{otherwise}.
\end{cases}
\]
Here, $\nu$ is a function satisfying $\nu(x)=0$ for $x\le0$, $\nu(x)=1$ for $x\ge 1$, and in addition, $\nu(x)+\nu(1-x)=1$ for $0\le x\le 1$.
For example, $\nu$ can be defined to be $\nu(x)=x^4(35-84x+70x^2-20x^3)$ for $x\in[0,1]$, which leads to $C^3$ functions of $\widehat\phi$ and $\widehat\psi$.

Then, the corresponding $H$ 
is given by
\[
H(e^{-2\pi i\xi}) = \begin{cases}
1 & |\xi|\le \frac16,\\
\cos(\frac{\pi}{2}\nu(6|\xi|-1)) & \frac16\le |\xi|\le \frac13, \\
0 & \frac13\le |\xi|\le \frac12.
\end{cases}
\]
We remark that $\xi \mapsto H(e^{-2\pi i\xi})$ is a $1$-periodic function and the Meyer wavelet function $\psi$ defined above satisfies
$\widehat\psi(2\xi) = -e^{-2\pi i \xi} \overline{H(e^{-2\pi i (\xi+\frac12)})}\widehat\phi(\xi)$. Hence the high-pass filter $G$ for $\psi$ is given
by $G(z)=-z\overline{H(-z)}$ with $z=e^{-2\pi i\xi}$. For any $k\in\mathbb{N}$, there exists $\nu$ such that $\widehat\phi$ and $\widehat\psi$ are
functions in $C^k(\mathbb R)$. Moreover, $\nu$ can be constructed to be $C^\infty$ so that
both $\widehat\phi$ and $\widehat\psi$ are functions in $C^\infty(\mathbb R)$ and their corresponding filters are functions in $C^\infty(\mathbb T)$.
 For more details about
Meyer wavelets, we refer to \cite{Daub:book} or \cite{Mal98}.


\subsubsection{Subband Decomposition for Discrete Data}

The two-scale relation in combination with downsampling as a simple data reduction strategy
is crucial for the efficient decomposition of data from some approximation space, say $V_0$.
We next formalize the decomposition of a function $f \in V_0 = V_{-1} \oplus W_{-1}$
in terms of the $Z$-transform.

For this, let the group of integer translations $\{T_n\}_{n \in \mathbb Z}$ acting on $L^2(\mathbb R)$
be defined by $T_n f(x) = f(x-n)$ for almost every $x \in \mathbb R$. Then, each function $f\in V_0$
can be expressed as
$$
   f = \sum_{n \in \mathbb Z} c_n T_n \phi.
$$
This enables us to associate with $f$ the values of the almost everywhere converging series
\[
Zf(z)=\sum_{n=-\infty}^\infty c_n z^n, \quad z \in \mathbb T.
\]

Letting now $H: \mathbb T \to \mathbb C $ be the low-pass filter of a localized multiresolution analysis
as specified above, the characterization of the subspace $V_{-1} \subset V_0$ can then be expressed as
\[
f \in V_{-1} \quad \Longleftrightarrow \quad Zf(z) = H(z)\left(\overline{ H(z)} Zf(z) +\overline{ H(-z)} Zf(-z)\right) \; \mbox{ for a.e. } z \in \mathbb T.
\]
This fact enables us to state a unified characterization of $V_{-1}$ and of $W_{-1} = V_0 \ominus V_{-1}$.
%
%
\begin{proposition}
\label{prop:projections}
Let $P_{V_{-1}}$ and $P_{W_{-1}}$ denote the orthogonal projection of $V_0$ onto $V_{-1}$ and $W_{-1}$,
respectively. Further, letting $H$ be defined as above, define $H_+$ to be the multiplication operator
given by $H_+F(z)=H(z)F(z)$, $H_-$ given by $H_- F(z) = H(-z) F(z)$, and $R_2$ the reflection operator
satisfying $R_2Zf(z)=Zf(-z)$. Then, we have
\[
Z P_{V_{-1}} f  =  H_+ (I+R_2) \overline {H_+} Zf
\quad \mbox{and} \quad
Z P_{W_{-1}} f = \overline {H_-} (I-R_2) {H_-} Zf.
\]
\end{proposition}

\begin{proof}
We first observe that the composition of down and up-sampling sets every other coefficient in the expansion
of $f$ to zero. After applying $Z$ this amounts to a periodization.

By definition, the projection onto $V_{-1}$ satisfies
$$
   Z P_{V_{-1}} f(z) =  H(z) \left(\overline{H(z)} Zf(z) + \overline{H(-z)} Zf(-z)\right) =  H_+ (I+R_2) \overline{H_+} Zf(z) \, .
$$
Similarly, the projection onto $W_{-1}$ is
\begin{align*}
  Z P_{W_{-1}} f(z) & = G(z) \left(\overline{G(z)} Zf(z) + \overline{G(-z)} Zf(-z)\right) \\
     & = -  z \overline {H(-z)} \left(-\overline{z} H(-z) Zf(z) +\overline{z} H(z) Zf(-z)\right) \\
     & =\overline{H(-z)} \left(H(-z)Zf(z) - H(z)Zf(-z)\right) \\
     & = \overline{H_-} (I-R_2) H_- Zf(z) \, .
\end{align*}
The proposition is proved.
\end{proof}

The relevance of these identities lies in the fact that $(I+R_2)\overline{H_+}Zf$ is an even function whereas $(I-R_2) H_- Zf$
is odd. Hence knowing every other coefficient in the series expansion is sufficient to determine the projection onto the
corresponding subband. Thus, in this case downsampling reduces the data without loss of information.


\subsubsection{$M$-Band Wavelets}

If instead of a dilation factor of $2$ in the two-scale relation, a factor of $M$ is used, $M-1$ wavelets are
necessary to complement the translates of $\phi$ to an orthonormal basis of the next higher resolution level.
In this situation, it is an $M\times M$ matrix which has to satisfy the orthogonality identity. Generalizing
the consideration in the previous subsection, let $\omega= e^{-2\pi i /M}$ and $R_M Z f(z) = Zf(\omega z)$
and the scaling mask $H_0$ for $\phi$ satisfy $\sum_{j=0}^{M-1} | H_0(\omega^j z)|^2 = 1$. We then define
the orthogonal
projection onto $V_{-1}$ in terms of the transform 
\[
ZP_{V_{-1}} f = H_0 \left(\sum_{j=0}^{M-1} R_M^j\right) \overline{H_0}Zf.
\]
For a proof that $P_{V_{-1}}$ is indeed a projection, see the more general statement in the next theorem.

We complement the filter $H_0$ by finding $H_n$ such that $(H_n(\omega^\ell z))_{n,\ell=0}^{M-1}$ is unitary for almost every
$z \in \mathbb T$. Once the wavelet masks $H_\ell, \ell=1,\ldots, M-1$ are constructed by matrix extension, the
wavelet functions $\psi_\ell, \ell=1,\ldots, M-1$ are given by $\widehat\psi_\ell(M\xi) = H_\ell(e^{-2\pi i \xi})\widehat\phi(\xi), \xi\in\bR$,
$\ell=1,\ldots, M-1$. It is well-known that then $\{\psi_\ell: \ell=1,\ldots,M-1\}$ generates an orthonormal wavelet basis for
$L^2(\bR)$.

One goal in $M$-band wavelet design is to choose $H_0$ and then to complete the matrix so that the
filters $H_n$ impart desirable properties on the associated scaling function and wavelets. In fact, one can construct
orthonormal scaling functions for any dilation factor $M\ge 2$ and the matrix extension technique applies for any
dilation factor $M\ge 2$. When $M>2$, the orthonormal bases can be built to be with symmetry (\cite{HKZ09,HanZhuang2010Mat, HanZhuang2013}).

In the same terminology as Proposition \ref{prop:projections}, we now have the following result that identifies
the orthogonal projections belonging to $M$-band wavelets.

\begin{theorem}
\label{theo:projections2}
Let $\{H_n\}_{n=0}^{M-1}$ be such that $(H_n(\omega^\ell z))_{n,\ell=0}^{M-1}$ is unitary for almost every $z \in \mathbb T$,
and let $\{P_{W_{-1,\ell}}\}_{\ell=0}^{M-1}$ be the operators defined by
$$
     Z P_{W_{-1,\ell}} f := H_\ell  \left(\sum_{j=0}^{M-1} R_M^j\right) \overline{H_\ell} Zf.
$$
Then $\{P_{W_{-1,\ell}}\}_{\ell=0}^{M-1}$ are mutually orthogonal projections (note that $P_{W_{-1},0}=P_{V_{-1}}$).
\end{theorem}

\begin{proof}
We first observe that by the assumed unitarity,
every row is normalized, and  each pair of rows is mutually orthogonal, i.e.,
$$
   \sum_{\ell=0}^{M-1} \overline{H_n(\omega^{\ell} z)}  H_m(\omega^{\ell} z)  = \delta_{n,m} \,,
$$
where $\delta_{n,m}=1$ if $n=m$ and $\delta_{n,m}=0$ otherwise.

Next, we show that each $P_{W_{-1,\ell}}$ is an orthogonal projection.
To begin with, we see that $P_{W_{-1,\ell}}$ is Hermitian because the sum $\sum_{j=0}^{M-1} R_M^j$ is
and this property is retained when it is conjugated by the multiplication operator $H_\ell$.
The fact that each $P_{W_{-1,\ell}}$ is idempotent
and that the projections are mutually orthogonal is due to the commutation relation
$$
    R_M \overline{H_\ell}= \overline{H_\ell(\omega \, \cdot)} R_M
$$
and because of the orthogonality of the rows in $(H_m(\omega^\ell z))_{m,\ell=0}^{M-1}$.
We have for $f \in L^2(\mathbb R)$ and almost every $z \in \mathbb T$,
\begin{align*}
Z P_{W_{-1,\ell}}  P_{W_{-1,k}} f(z)&= H_\ell \sum_{m=0}^{M-1} R_M^m  \overline {H_\ell} H_k \sum_{n=0}^{M-1} R_M^n \overline{H_k} Z f(z)\\
   & =    H_\ell \sum_{m=0}^{M-1}   \overline{H_\ell(\omega^{m}z)} H_k(\omega^{m}z) R_M^m \sum_{n=0}^{M-1} R_M^n \overline{H_k} Z f(z)\\
   & =    H_\ell \sum_{m=0}^{M-1}   \overline{H_\ell(\omega^{m}z)} H_k(\omega^{m}z) \sum_{n=0}^{M-1} R_M^{n} \overline{H_k} Z f(z)\\
   & =   \delta_{\ell,k}  H_\ell  \sum_{n=0}^{M-1} R_M^n \overline{H_k} Z f(z)= \delta_{\ell,k} Z P_{W_{-1,\ell}} f(z)\, .
\end{align*}
This finishes the proof.
\end{proof}


\subsection{From Group-Based to Cone-Adapted Shearlets}
\label{subsec:Shearlets}

In contrast to wavelets, shearlet systems are based on three operations: scaling, translation, and
shearing; the last one to change the orientation of those anisotropic functions. Letting the
(parabolic) scaling matrix $A_j$ be defined by
\[
A_{j} = \left(\begin{array}{cc} 4^j & 0 \\ 0  & 2^{j}  \end{array}\right), \quad j \in \bZ,
\]
and the shearing matrix $S_k$ be
\[
S_k =  \left(\begin{array}{cc}  1 & -k \\ 0 & 1  \end{array}\right), \quad k \in \bZ.
\]
Then, for some generator $\psi \in L^2(\bR^2)$, the {\em group-based shearlet system} is defined by
\[
\{2^{\frac{3j}{2}} \psi(S_{k} A_{j} \, \cdot \, - m) \: : \: j,k \in \bZ, m \in \bZ^2\}.
\]

Despite the nice mathematical properties -- this system can be regarded as arising from a representation
of a locally compact group, the shearlet group -- group-based shearlet systems suffer from
the fact that they are biased towards one axis which prevents a uniform treatment of directions. Cone-adapted
shearlet systems circumvent this problem, by utilizing a particular splitting of the frequency domain into a
vertical and horizontal part. For this, we set $A_j^h := A_j$, $S_k^h := S_k$,
\[
A_{j}^v = \left(\begin{array}{cc} 2^{j} & 0 \\ 0  & 4^j  \end{array}\right),
\quad \mbox{and} \quad
S_k^v =  \left(\begin{array}{cc}  1 & 0 \\ -k & 1  \end{array}\right), \quad j, k \in \bZ.
\]
Given a scaling function $\phi \in L^2(\bR^2)$ and some $\psi \in L^2(\bR^2)$, the {\em cone-adapted
shearlet system} is defined by
\begin{eqnarray*}
\{\phi(\cdot-m) : m \in \bZ^2\}
& \cup & \{2^{3j/2} \psi(S^h _kA^h_{j}\, \cdot \, -m) : j \ge 0, |k| \leq  2^{j}, m \in \bZ^2 \}\\
& \cup & \{2^{3j/2} \tilde{\psi}({S}^v_k A^v_{j}\, \cdot \, -m) : j \ge 0, |k| \leq  2^{j} , m \in \bZ^2 \},
\end{eqnarray*}
where $\tilde{\psi}(x_1,x_2) = \psi(x_1,x_2)$. For more details on shearlets, we refer to \cite{KL12}.

Gabor shearlets will also be constructed first as group based systems, and then in a cone-adapted
version. However, in contrast to other constructions, we aim at low redundancy in the group-based system and
avoid increasing it in the cone adaptation.


\subsection{Gabor Frames}
\label{subsec:GaborFrames}

Like the previous systems, Gabor systems are based on translation and modulation. As usual, we
denote the modulations on $L^2(\mathbb R)$ by $M_m f(\xi) = e^{2\pi i m \xi} f(\xi)$.

By definition of tightness, a square-integrable function $w: \mathbb R \to \mathbb C$ is the window of a
$b^{-1}$-tight Gabor frame, if it is unit norm and for all $f \in L^2(\mathbb R)$,
\[
   \| f\|^2 = b \sum_{m,n\in \mathbb Z} |\langle f,  M_{m b} T_n w\rangle |^2 \, .
\]
For more details on Gabor systems, we refer the reader to \cite{Gro01}.

Various ways to construct such a window function $w$ are known. We recall
a construction of a $b^{-1}$-tight Gabor frame with $b^{-1}>1$ arbitrarily close to 1  \cite{DGM86}.

\begin{example}
\label{exa:Gabor}
Let $\nu$ be in $C^\infty(\mathbb R)$ and $\nu(x)=0$ for $x \le 0$, $\nu(x)=1$ for $x \ge 1$ and $\nu(1-x)+\nu(x)=1$.  Let $w(x):=(\nu((1/2+\varepsilon-|x|)/2\varepsilon))^{1/2}$,
$x\in\mathbb R$.
Then, it is easy to show that $w$ is a smooth
function with support belonging to $[-1/2-\varepsilon, 1/2+\varepsilon]$
for any  $0<\varepsilon \le 1/2$, $\|w\|^2=1$ and $\sum_n|T_n w |^2=1$. Consequently,
if $b=(1+2\epsilon)^{-1}$, then $\{M_{mb}T_{n} w: m , n \in \mathbb Z\}$ defines a $b^{-1}$-tight
Gabor frame.

%
%
%
%
%

 \end{example}


\subsection{Redundancy}
\label{subsec:redundancy}

Since we cannot achieve (P1), but would like to approximate this property, besides the classical frame definition,
we also require a notion of redundancy.
The first more refined definition of redundancy besides the classical ``number of elements divided by
the dimension'' definition was introduced in \cite{BCK11}. The extension of this definition
to the infinitely dimensional case can be found in \cite{CCH}. Since this work is not intended for publication,
we make this subsection self-contained.

We start by recalling a redundancy function, which provides a means to measure the concentration of the frame
close to one vector.
%
If $\{\varphi_i\}_{i \in I}$ is a frame for a real or complex Hilbert space $\cH$ without any zero vectors, and
let $\bS = \{x \in \cH : \|x\|=1\}$, then for each $x \in \bS$, the associated {\em redundancy function}
$\cR : \bS \to \bR^+\cup \{\infty\}$ is  defined by
\[
\cR(x) = \sum_{i \in I} \norm{\varphi_i}^{-2} \left| \left\langle x,\varphi_i\right\rangle \right|^2.
\]

Taking the supremum or the infimum over $x$ in this definition gives rise to the so-called upper and lower redundancy, which is in fact the
upper and lower frame bound of the associated normalized frame,
\[
\cR^+ = \sup_{x\in \bS} \cR(x)
\quad \mbox{ and } \quad
\cR^- = \inf_{x\in\bS} \cR(x).
\]
For those values, it was proven in \cite{BCK11}
that in the finite-dimensional situation, the upper redundancy provides a means to measure the minimal number of
linearly independent sets, and the lower redundancy is related to the maximal number of spanning sets, thereby
linking analytic to algebraic properties.


It is immediate to see that an orthonormal basis satisfies $\cR^- = \cR^+ = 1$, and a unit norm $A$-tight
frame $\cR^- = \cR^+ = A$. This motivates the following definition.

\begin{definition}
A frame $\{\varphi_i\}_{i= \in I}$ for a real or complex Hilbert space has a {\em uniform redundancy}, if
$
\cR^- = \cR^+ \, ,
$
and if it is unit norm and $A$-tight, then we say that it has {\em redundancy}
$A$.
\end{definition}
In the sequel, we will use the redundancy to determine
to which extent (P1) is satisfied.


\section{Group-Based Gabor Shearlets}
\label{sec:GroupGaborShearlets}

Let us start with an informal description of the construction of Gabor shearlets in a special case with
the goal to first provide some intuition for the reader.

Generally speaking, the shearlet construction in this paper is a Meyer-type modification of a
multiresolution analysis based on the Shannon shearlet scaling function
$
   \widehat \Phi_{0,0,0} = \chi_{K},
$
where $K= \{\xi \in \mathbb R^2: | \xi_1 | \le 1 \mbox{ and } |\xi_2/\xi_1 | \le 1/2 \}$. For an illustration,
we refer to Figure~\ref{fig:Shannonshear}.
\begin{figure}[h]
\begin{center}
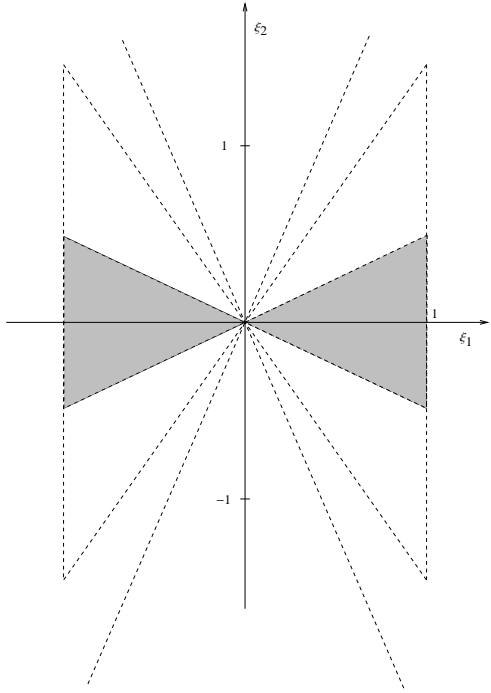
\end{center}
\caption{Support of the scaling function belonging to group-based Shannon shearlets (as well as the group-based Gabor shearlets)  in the frequency domain.
Additional lines indicate the boundaries of the support for sheared scaling functions.}
\label{fig:Shannonshear}
\end{figure}

It is  straightforward to verify that chirp modulations
\[
  \widehat \Phi_{0,0,m}(\xi) =  \chi_{K}(\xi) e^{2\pi i m_2 \xi_2/\xi_1} e^{\pi i m_1 \xi_1^3/|\xi_1|}, \, \, m=(m_1,m_2)\in{\mathbb Z}^2,
\]
define an orthonormal system $\{\Phi_{0,0,m}: m \in \mathbb Z^2\}$, while the use of the usual modulations
\[
  \widehat \Upsilon_{0,0,m}(\xi) = \chi_{K}(\xi) e^{2\pi i m_2 \xi_2} e^{\pi i m_1 \xi_1} \,
\]
gives a 2-tight frame $\{\Upsilon_{0,0,m}: m \in \mathbb Z^2\}$ for its span. The same is true when
the modulations are augmented with shears, $ \widehat \Phi_{0,k,m}(\xi) = \widehat \Phi_{0,0,m}(\xi_1,\xi_2-k\xi_1)$
and likewise for $\widehat \Upsilon_{0,k,m}$, in order to form the orthonormal or tight systems
$\{\Phi_{0,k,m}: k \in \mathbb Z, m \in \mathbb Z^2\}$ or $\{\Upsilon_{0,k,m}: k \in \mathbb Z, m \in \mathbb Z^2\}$,
respectively. Because both systems are unit-norm, the tightness constant is a good measure for redundancy as detailed
in Subsection \ref{subsec:redundancy}, indicating that chirp modulations are preferable from this point
of view. Incorporating parabolic scaling preserves those properties.

In a second step (Section \ref{sec:ConeGaborShearlets}) the strategy of shearlets is followed to derive a
cone-adapted version, which provides the property of a uniform treatment of directions necessary for optimal
sparse approximation results. A further necessary ingredient for optimal sparsity are good decay properties.
We show that a combination of Gabor frames, Meyer wavelets and a change of coordinates provides
smooth alternatives for the characteristic function, yet with still near-orthonormal shearlet systems
that are similar to the Shannon shearlet we described.


\subsection{Construction using Chirp Modulations}
\label{subsec:GroupConstruction}

To begin the shearlet construction, we examine an alternative group of translations acting as chirp modulations
in the frequency domain. These modulations do not correspond to the usual Euclidean translations, but for
implementations in the frequency domain this is not essential. In the following, we use the notation
$\bR^* := \bR \setminus \{0\}$.

\begin{definition}
Let $\gamma(\xi):=(\gamma_1(\xi),\gamma_2(\xi))$ with $\gamma_1(\xi):=\frac12\xi_1^3/|\xi_1|=\frac12\sgn(\xi_1) \xi_1^2$ and $\gamma_2(\xi):
=\frac{ \xi_2}{\xi_1}$ for $\xi=(\xi_1,\xi_2)\in \bR^*\times\bR$. We define the two-dimensional \emph{chirp-modulations}
$\{X_\beta: \beta \in \mathbb R^2\}$ by
\[
    X_{\beta} \widehat f(\xi) = e^{2\pi i \beta_1 \gamma_1(\xi)} e^{2\pi i \beta_2 \gamma_2(\xi)} \widehat f(\xi),\quad \xi \in \mathbb R^* \times \mathbb R \, .
\]
\end{definition}

We emphasize that the set with $\xi_1=0$ is excluded from the domain, which does not cause problems
since it has measure zero.

Next, notice that the point transformation $\gamma$ has a Jacobian of magnitude one and is a bijection
on $\mathbb R^* \times \mathbb R$. Therefore, it defines a unitary operator $\Gamma$ according to
$$
  \Gamma \widehat f(\xi) = \widehat f(\gamma(\xi)), \quad\xi\in\bR^*\times\bR.
$$

As discussed in Subsection \ref{subsec:Shearlets}, the shear operator is a further ingredient of shearlet systems.
By abuse of notation, for any $s \in \mathbb R$, we will also regard $S_s$ as an operator, that is
\[
S_{s} \widehat f( \xi_1,\xi_2) = \widehat f(\xi_1, \xi_2- s \xi_1).
\]

The benefit of choosing the chirp-modulations is that shearing and modulation satisfy the well-known Weyl-Heisenberg
commutation relations. The proof of the following result is a straightforward calculation, hence we omit it.

\begin{proposition}
\label{prop:Weyl}
For $s \in \mathbb R$ and $\beta  \in \mathbb R^2$,
$$
  S_s X_\beta  = e^{- 2\pi i \beta_2 s} X_\beta S_s  \, .
$$
\end{proposition}

The last ingredient is a scaling operator which gives parabolic scaling.
Again abusing notation, we write the dilation operator with $A_j$.
%
For $j \in \mathbb Z$, we let $A_j$ be the dilation operator acting on $f \in L^2(\mathbb R^2)$ by
\[
  A_j \widehat f (\xi_1,\xi_2) = 2^{-3j/2} \widehat f(2^{-2 j} \xi_1, 2^{-j} \xi_2)\,
\]
for almost very $\xi=(\xi_1,\xi_2) \in \mathbb R^2$.

The last ingredient to define group-based Gabor shearlets are the generating functions to which
those three operators are then applied.
%
For this, let $\phi$ be an orthogonal scaling function of
a $16$-band multiresolution analysis in $L^2(\mathbb R)$, with associated orthonormal wavelets
$\{\psi_\ell\}_{\ell=1}^{15}$, and let $w$ be the unit norm window function of a $b^{-1}$-tight
Gabor frame $\{M_{m_2b}T_{k} w: m_2, k \in \mathbb Z\}$ for $L^2(\mathbb R)$. Then we define the generators
\[
\widehat{\Phi_{0,0,0}}:=\Gamma\widehat\phi\otimes w
\quad \mbox{and} \quad
\widehat{\Psi^\ell_{0,0,0}}:=\Gamma\widehat\psi_\ell\otimes w, \;\; \ell=1,\ldots,15.
\]
in $L^2(\mathbb R^2) = L^2(\mathbb R) \otimes  L^2(\mathbb R)$, based on which we now define group-based Gabor shearlets.

\begin{definition}
\label{defi:GroupGaborShearlets}
Let $\Phi_{0,0,0}$ and $\Psi^\ell_{0,0,0}$, $\ell=1,\ldots,15$, and $w$ be defined as above. Let $j_0\in\bZ$. Then the {\em group-based
Gabor shearlet system} is defined by
\[
\cG\cG\cS_{j_0}(\phi,\{\psi_\ell\}_{\ell=1}^{15};w) :=  \{\Phi_{j_0,k,m} : k \in \bZ, m \in \bZ^2\} \cup \{\Psi_{j,k,m}^\ell : j,k \in \bZ, j\ge j_0, m \in \bZ^2, \ell=1,\ldots,15\} \subseteq L^2(\mathbb R^2),
\]
where
\[
\begin{aligned}
\widehat{\Phi_{j,k,m}}(\xi)&=A_jX_{(m_1,m_2b)}S_{k}\widehat\Phi_{0,0,0}\\&=
2^{-3j/2}\widehat\phi(2^{-4j}\gamma_1( \xi)) w(2^{j}\gamma_2(\xi)-k)
e^{2\pi i m_12^{-4j}\gamma_1(\xi)}e^{2 \pi i m_2 b
2^{j}\gamma_2(\xi)},
\end{aligned}
\]
and
\[
\begin{aligned}
  \widehat{\Psi^\ell_{j,k,m}}(\xi)&=A_jX_{(m_1,m_2b)}S_k\widehat\Psi^\ell_{0,0,0}
\\&=
2^{-3j/2}\widehat\psi_\ell(2^{-4j}\gamma_1(\xi)) w (2^{j}\gamma_2(\xi)-k)
e^{2\pi i m_12^{-4j}\gamma_1(\xi)}e^{2 \pi i m_2 b
2^{j}\gamma_2(\xi)}.
\end{aligned}
\]
\end{definition}

The particular choice of dilation factors in the first and second coordinate comes
from the need for parabolic scaling and integer dilations. The motivation is that the regularity
of the singularity in the cartoon-like model is $C^2$, and if the generator satisfies $width =length^2$
one can basically linearize the curve inside the support with controllable error by the Taylor
expansion. Since we utilize a different group operation, it is not immediately clear
which scaling leads to the size constraints $width  = length^2$. An integer value of $j$ requires $4j = j^2$, so $j=4$.
Then one
considers the intertwining relationship between the dilation operator $A_4$ and
the standard one-dimensional dyadic dilation $D$ to deduce
$
   A_4 \Gamma  = \Gamma D^{-16} \otimes D^4 \, ,
$
which explains the choice of $M=16$ bands.





\subsection{MRA Structure}
\label{subsec:GroupMRA}

One crucial question is whether the just introduced system is associated with an MRA structure.
As a first step, we define associated scaling and wavelet spaces.

\begin{definition}
\label{defi:GroupGaborShearletMRA}
Let $\Phi_{j,k,m}$ and $\Psi^\ell_{j,k,m}$, $j, k \in \bZ, m \in \bZ^2, \ell=1,\ldots,15$ be defined as in Definition~\ref{defi:GroupGaborShearlets}. For each $j \in \bZ$,
the {\em scaling space} $V_j$ is the closed subspace
\[
V_j =  \overline{\mathrm{span}}\{\Phi_{j,k,m} : k \in \bZ, m \in \bZ^2\} \subseteq L^2(\mathbb R^2),
\]
and the associated {\em wavelet space} $W_j$ is defined by
\[
\begin{aligned}
W_j =  \overline{\mathrm{span}} \{\Psi_{j,k,m}^\ell : k \in \bZ, m \in \bZ^2, \ell=1,\ldots,15\}.
\end{aligned}
\]
\end{definition}

Next, we establish that the group-based Gabor shearlet system is indeed associated with an MRA structure,
and analyze how close it is to being an orthonormal basis.

\begin{theorem}
\label{thm:subspace}
Let $\Phi_{j,k,m}$ and $\Psi^\ell_{j,k,m}$, $j, k \in \bZ, m \in \bZ^2, \ell=1,\ldots,15$ be defined as in Definition~\ref{defi:GroupGaborShearlets} and let $\{V_j\}_{j\in \mathbb Z}$ and $\{W_j\}_{j\in \mathbb Z}$ be the associated scaling
and wavelet spaces as defined in Definition~\ref{defi:GroupGaborShearletMRA}. Then, for each $j \in \mathbb Z$, the family $\{\Phi_{j,k,m}: k \in \mathbb Z, m \in \mathbb Z^2\}$
is a unit norm $b^{-1}$-tight frame for $V_j$, and $\{\Psi^\ell_{j,k,m}: k \in \mathbb Z, m
\in \mathbb Z^2,  \ell=1, \ldots, 15\}$ forms a unit-norm $b^{-1}$-tight frame for $W_j$.
\end{theorem}

\begin{proof}
We first verify that the scaling function generates a $b^{-1}$-tight frame for a closed subspace of
$L^2(\mathbb R^2)$. By Proposition \ref{prop:Weyl}, the operator $\Gamma$ intertwines shears and
translations in the second component,
\[
S_k\Gamma\widehat f(\xi)= \Gamma \widehat f(\xi_1,\xi_2-k\xi_1)=\widehat f(\gamma_1(\xi),\gamma_2(\xi)-k) \,.
\]
Moreover, it intertwines chirp modulations with standard modulations. The overall dilation is irrelevant
because $A_j$ is unitary, so we can set $j=0$ for simplicity. Therefore, it is enough to prove that
$\{M_{m_1} \widehat \phi \otimes M_{m_2 b} T_{k}w \}$ defines a $b ^{-1}$-tight frame for $\Gamma^{-1}(\widehat{V}_j)$. This
follows from the fact that $w$ is the unit norm window function of a $b^{-1}$-tight Gabor frame and
from $\phi$ being an orthonormal scaling function.
Since the subspaces $\phi\otimes L^2(\mathbb R), \psi_\ell \otimes L^2(\mathbb R), \ell=1,\ldots,15$
are mutually orthogonal, and the functions $\{\phi,\psi_\ell: \ell=1,\ldots, 15\}$ satisfy a two-scale relation of an MRA
with dilation factor $M=16$ in $L^2(\mathbb R)$, the claim follows.
\end{proof}

\begin{theorem} The scaling and wavelet subspaces $V_0$ and $W_0$ as defined in Definition~\ref{defi:GroupGaborShearletMRA} satisfy the two-scale relation
$$
  V_ 0 \oplus W_0 = A_4 V_0 \, .
$$
\end{theorem}
\begin{proof}
We note that the functions
\[
\widehat \phi\otimes  w
\quad \mbox{and} \quad
\widehat \psi_\ell\otimes w, \ell=1, \dots, 15
\]
are orthogonal by assumption, and the orthogonality remains under the usual modulations in the first component.
On the other hand, the window function in the second component forms a tight Gabor  frame under translations and modulations,
so each of the tensor products generates a tight frame for its span.

Since the subspaces $\phi\otimes L^2(\mathbb R), \psi_\ell \otimes L^2(\mathbb R), \ell=1,\ldots,15$
are mutually orthogonal, and the functions $\{\phi,\psi_\ell\}$ satisfy a two-scale relation of an MRA
with dilation factor $M=16$ in $L^2(\mathbb R)$, the claim follows.
\end{proof}



Since implementations only concern a finite number of scales, the following result becomes important.
It is an easy consequence of Theorem~\ref{thm:subspace}.

\begin{corollary}\label{cor:subspace}
The group-based Gabor shearlet system $\cG\cG\cS_{j_0}(\phi,\{\psi_\ell\}_{\ell=1}^{15}; w)$
as defined in Definition \ref{defi:GroupGaborShearlets} for any $j_0\in\bZ$, or
 the system $\{\Psi^\ell_{j,k,m}: j, k \in \mathbb Z, m \in \mathbb Z^2, \ell=1, \ldots, 15\}$
forms a unit-norm $b ^{-1}$-tight frame for $L^2(\bR^2)$, and consequently it has uniform redundancy $\mathcal{R}^-
= \mathcal{R}^+ = b^{-1}$.
\end{corollary}


\section{Cone-Adapted Gabor Shearlets}
\label{sec:ConeGaborShearlets}

The construction of nearly orthonormal cone-adapted Gabor shearlets is based on complementing a core subspace $V_0$ which has the usual MRA properties for $\Lp{2}$ under scaling
with a dilation factor of $16$.
The isometric embedding of $V_0$ in $V_1$ proceeds in 3 steps:
\begin{enumerate}
\item $V_1$ is split into a direct sum of two coarse-directional subspaces, $V_1^h$ and $V_1^v$, corresponding to horizontally and vertically aligned details, respectively.
\item Each of these two coarse-directional subspaces is split into a direct sum of high and low pass components.
The low-pass subspaces $V_{0}^h$ and $V_{0}^v$ combine to $V_{0}=V_{0}^h \oplus V_{0}^v$.
\item The high pass components are further split into subspaces with a finer directional resolution obtained from shearing.
\end{enumerate}

The first step in the process of constructing the cone-adapted shearlets is a splitting between features that are
mostly aligned in the horizontal or in the vertical direction. The shearlets then refine this coarse splitting.


\subsection{Cone adaptation}
\label{subsec:ConeSplitting}

In addition to filters which restrict to cones in the frequency domain, we introduce
quarter rotations for the splitting of horizontal and vertical features.
This enables us to define two mutually orthogonal closed subspaces
containing functions with support near the usual cones for horizontal and vertical
components. As in the case of wavelets, the main goal of this construction is that
the smoothness of a function
 in the frequency domain is not substantially degraded by the projection onto
 the subspaces.

Again, we use standard filters from wavelets in our construction.
For this, we define a version of the Cayley transform $\zeta (\xi) = \frac{1+i\xi}{1-i\xi}$, which maps $\xi \in \mathbb R$ to
the unit circle $\mathbb T = \{z \in \mathbb C: |z|=1\}$. The inverse map is defined on $\mathbb T \setminus \{-1\}$,
$\zeta^{-1}(z) = i \frac{1-z}{1+z}$. We use the map $\zeta$ to lift polynomial filters on $\mathbb T$ to
rational filters on $\mathbb R$.

\begin{lemma}
Let $H: \mathbb T \to \mathbb C$ satisfy $|H(z)|^2 + |H(-z)|^2 = 1$
for all $z \in \mathbb T$, then $\tilde H(\xi) := H(\zeta(\xi))$ is
a function on $\mathbb R$ which satisfies
$$
   |\tilde H (\xi)|^2 + |\tilde H(-1/\xi)|^2 = 1 \, .
$$
\end{lemma}

\begin{proof}
The Cayley transform intertwines the reflection $\xi \mapsto -1/\xi$ on $\mathbb R$ with the reflection about the origin,
because
$$
   \zeta(-{1}/{\xi}) = \frac{1-i/\xi}{1+i/\xi} = \frac{1+i\xi}{-1+i\xi} = -\zeta(\xi) \, .
$$
Thus, the property of $\tilde H$ is a direct consequence of this coordinate transformation.
\end{proof}

We observe that if $H(z)$ has $N-1$ vanishing derivatives at $z=-1$, $H(-1)=H'(-1) = \cdots = H^{(N-1)}(-1)=0$, then
$\tilde H(\xi)$ decays as $\xi^{-N}$ at infinity.

\begin{definition}
\label{defi:Hs}
Let $H: \mathbb T \to \mathbb C$ satisfy the Smith-Barnwell
condition $|H(z)|^2 + |H(-z)|^2 = 1$ for all $z \in \mathbb T$. Its associated filter operators
 $H_+$, $\overline{H_+}$, $H_-$ and
$\overline{H_-}$, are defined to be the multiplicative operators with the Fourier transform of any $f \in
\Lp{2}$ in the frequency domain according to $H_+ \widehat{f}(\xi) =
H(\zeta(\xi_2/\xi_1)) \widehat{f}(\xi)$ and $H_- \widehat{f} (\xi) =
H(-\zeta(\xi_2/\xi_1)) \widehat{f}(\xi)$, the overbar denoting
multiplication with the complex conjugate. We denote $R$  to be the
rotation operator on $\Lp{2}$  given by $R \widehat{f} (\xi_1,\xi_2) =  \widehat f(\xi_2,
-\xi_1)$.
\end{definition}

This allows us to introduce a pair of complementary orthogonal projections, which
split the group based Gabor shearlets into a vertical
and a horizontal part to balance the treatment of directions. The design
of these projections is inspired by the description of smooth projections in \cite{WH96}.

We start the construction with
isometries associated with the
vertical and horizontal cone, which we denote by ${\mathcal
C}_v$ and ${\mathcal C}_h$, respectively. By the set inclusion,
$L^2({\mathcal C}_v)$ and $L^2({\mathcal C}_h)$ naturally embed isometrically 
in
$L^2(\mathbb R^2)$. We denote these embeddings by $\iota_v: \iota_v
f(\xi) = f(\xi)$ if $\xi \in C_v$ and $\iota_vf(\xi)=0$ otherwise and
similarly for $\iota_h$. We wish to find isometries that do not create discontinuities.

\begin{theorem}
Let $H: \mathbb T \to \mathbb C$ satisfy $|H(z)|^2 + |H(-z)|^2 = 1$ for all $z \in \mathbb T$,
and let $H_+$, $\overline{H_+}$, $H_-$ and $\overline{H_-}$ be defined as in Definition~\ref{defi:Hs}.
Let ${\mathcal C}_v = \{ x \in \mathbb R^2: |x_2| \ge |x_1| \}$ and ${\mathcal C}_h = \mathbb R^2\setminus {\mathcal C}_v$,
then
the map $\Xi_v: L^2({\mathcal C}_v) \to L^2(\mathbb R^2)$ given by
$$
  \Xi_v f =\overline{ H_-} \left (I-\frac{1+i}{2} R - \frac{1-i}{2} R^3 \right) \iota_v f \, .
$$
is an isometry, and so is the map $\Xi_h: L^2({\mathcal C}_h) \to L^2(\mathbb R^2)$,
$$
  \Xi_h f =  H_+( I + \frac{1+i}{2}R + \frac{1-i}{2} R^3) \iota_h f \, .
$$
Moreover, the range of $\Xi_v$ is the orthogonal complement of the range of $\Xi_h$ in $L^2(\mathbb R^2)$.  
\end{theorem}

\begin{proof}
We begin by showing that $\Xi_v$ and $\Xi_h$ are isometries.
The space $L^2({\mathcal C}_v)$ splits into even and odd functions. After embedding in $L^2(\mathbb R^2)$ these functions
then satisfy $R^2 \iota_v f = \iota_v f$ or $R^2 \iota_v f = - \iota_v f$, respectively.

By the definition of $R$,
the operator $ I- \frac{1+i}{2}R - \frac{1-i}{2} R^3$ maps
even $\iota_v f$ to
$$
  \left (I- \frac{1+i}{2}R - \frac{1-i}{2} R^3 \right ) \iota_v f \\
   = \left ( \frac 1 2I+\frac12 R^2 - \frac 1 2 R - \frac 1 2 R^3\right ) \iota_v f \, ,
$$
which implies
that it is an eigenvector of $R$,
$ R ( \frac 1 2I +\frac12 R^2- \frac 1 2 R - \frac 1  2 R^3) \iota_v f = -  ( \frac 1 2I +\frac12 R^2- \frac 1 2 R - \frac 1 2 R^3) \iota_v f  $
and for odd $f$
$$
  \left  (I- \frac{1+i}{2}R - \frac{1-i}{2} R^3\right ) \iota_v f \\
  = \left ( \frac 1 2 I- \frac 1 2 R^2 - \frac i 2 R + \frac i 2 R^3\right ) \iota_v f \, ,
$$
which gives
$R ( \frac 1 2I - \frac 1 2 R^2 - \frac i 2 R + \frac i 2 R^3) \iota_v f  = i ( \frac 1 2 I- \frac 1 2 R^2 - \frac i 2 R + \frac i 2 R^3) \iota_v f $.

Similarly,
the operator
$( I + \frac{1+i}{2}R + \frac{1-i}{2} R^3)$ maps the even
functions  into functions that are invariant under $R$, whereas the odd functions
give eigenvectors of $R$ corresponding to eigenvalue $-i$. We verify that for even $\iota_h f$,
$$
   \left (I+ \frac{1+i}{2}R + \frac{1-i}{2} R^3\right ) \iota_h f
   = \left ( \frac 1 2 I + \frac 1 2 R^2 + \frac 1 2 R + \frac 1 2 R^3\right ) \iota_h f.
$$
Hence we get the eigenvalue equation
$R ( \frac 1 2 I+ \frac 1 2 R^2 + \frac 1 2 R + \frac 1 2 R^3) \iota_h f =  ( \frac 1 2 I + \frac 1 2 R^2 + \frac 1 2 R + \frac 1 2 R^3) \iota_h f$.
Analogously, for odd $f$,
$$
   \left (I+ \frac{1+i}{2}R + \frac{1-i}{2} R^3\right ) \iota_h f
   = \left ( \frac 1 2 I- \frac 1 2 R^2 + \frac i 2 R - \frac i 2  R^3\right ) \iota_h f
$$
which yields $R ( \frac 1 2 I - \frac 1 2 R^2 + \frac i 2 R - \frac i 2  R^3) \iota_h f = (-i) ( \frac 1 2 I - \frac 1 2 R^2 + \frac i 2 R - \frac i 2  R^3) \iota_h f$.

Since $R$ is unitary, the eigenvector equations imply that the
orthogonality between even and odd functions is preserved by the embedding followed by the symmetrization
with $( I + \frac{1+i}{2}R + \frac{1-i}{2} R^3)$ or $( I - \frac{1+i}{2}R - \frac{1-i}{2} R^3)$. Thus,
the identity
$$
  \left \| \left ( I - \frac{1+i}{2}R - \frac{1-i}{2} R^3\right ) \iota_v f \right \|^2_{L^2(\mathbb R^2)}
  = 2 \| f\|^2_{L^2({\mathcal C}_v)} \, \mbox{ for all  } f \in L^2({\mathcal C_v})
$$
can be verified by checking it separately for even and odd functions.
Next, multiplying by $\overline{H_-}$ and using that $|H_-|^2 + R^{-1} |H_-|^2 R = |H_-|^2 + |H_+|^2=1$
gives by the orthogonality of $R\iota_v f$ and $\iota_v f$ the isometry
\begin{align*}
  \left \| \overline{H_-} \left ( I - \frac{1+i}{2}R - \frac{1-i}{2} R^3\right ) \iota_v f \right \|^2_{L^2(\mathbb R^2)}
    & = \|\overline{H_-} \iota_v f\|^2_{L^2(\mathbb R^2)} + \left \| H_- \left (\frac{1+i}{2} R  + \frac{1-i}{2} R^3\right ) \iota_v f\right \|^2_{L^2(\mathbb R^2)} \\
   & = \|H_- \iota_v f\|^2_{L^2(\mathbb R^2)} + \left \| H_+ \left (\frac{1+i}{2}   + \frac{1-i}{2} R^2\right ) \iota_v f\right \|^2_{L^2(\mathbb R^2)} \\
   &= \|H_- \iota_v f\|^2_{L^2(\mathbb R^2)} + \| H_+ \iota_v f\|^2_{L^2(\mathbb R^2)} = \| f\|^2_{L^2(\mathcal C_v)} \, .
\end{align*}
The same proof applies to $\Xi_h$.

To show that the ranges are orthogonal complements of each other, we define 
the orthogonal projections $P_v = \Xi_v \Xi_v^*$ and $P_h = \Xi_h \Xi_h^*$.
We first establish that these projections have the more convenient expressions
$$
   P_h  = H_+ \left (I+\frac{1+i}{2} R + \frac{1-i}{2} R^3 \right) \overline{H_+}
\quad \mbox{and} \quad
 P_v = \overline{H_-} \left (I-\frac{1+i}{2} R - \frac{1-i}{2} R^3 \right)  H_- \, .
$$
To this end, we note that if $M_v$ is the multiplication operator $M_v f(\xi) = \chi_{{\mathcal C}_v} f(\xi)$
with $\chi_{{\mathcal C}_v}$ the characteristic function of ${\mathcal C}_v$, and similarly for
$M_h$, $M_h f(\xi) = \chi_{{\mathcal C}_h}(\xi)$, then by definition
$$
    P_h = \Xi_h \Xi_h^* = H_+ (I+ \frac{1+i}{2} R  + \frac{1-i}{2}R^3 ) M_h (I+\frac{1+i}{2}R + \frac{1-i}{2} R^3) \overline{H_+}  \, .
$$
We simplify this expression using that $M_h R = R M_v $, $M_v + M_h = I$ and $R^2 M_h = M_h R^2$,
which gives the identities
$$
  (\frac{1+i}{2} R  + \frac{1-i}{2}R^3 ) M_h 
     + M_h (\frac{1+i}{2} R  + \frac{1-i}{2}R^3 )
      = (\frac{1+i}{2} R  + \frac{1-i}{2}R^3 ) M_h + (\frac{1+i}{2} R  + \frac{1-i}{2}R^3 ) M_v 
      = \frac{1+i}{2} R  + \frac{1-i}{2}R^3 
$$
and 
$$
   (\frac{1+i}{2} R  + \frac{1-i}{2}R^3 ) M_h (\frac{1+i}{2} R  + \frac{1-i}{2}R^3 )
    =   (\frac{1+i}{2} R  + \frac{1-i}{2}R^3 )^2 M_v = M_v \, .
$$
Inserting this in the expression for $P_h$ results in
\begin{align*}
    &H_+ (I+ \frac{1+i}{2} R  + \frac{1-i}{2}R^3 ) M_h (I+\frac{1+i}{2}R + \frac{1-i}{2} R^3) \overline{H_+} \\
    & = H_+ (M_h + (\frac{1+i}{2} R  + \frac{1-i}{2}R^3 ) 
                  + M_v ) \overline{H_+} 
                  = H_+ (I +  \frac{1+i}{2} R  + \frac{1-i}{2}R^3 ) \overline{H_+}  \, .
\end{align*}
The identities for $P_v$ are completely analogous.

Finally, we show that the two orthogonal projections are complementary. To this end, we use
$$
   P_h = H_+\overline{H_+}  +H_+ \overline{H_-}  (\frac{1+i}{2} R + \frac{1-i}{2} R^3 )
$$
and
$$
  P_v =  \overline{H_-} H_-    - \overline{H_-}  H_+ (\frac{1+i}{2} R + \frac{1-i}{2} R^3)
$$
which gives the identity after elementary cancellations and $H_+ \overline{H_+}  + \overline{H_-} H_-  = I$.
Since $P_h$ is by definition an orthogonal projection, $P_v=I-P_h$ is the complementary one.
Thus, the ranges of $\Xi_h$ and $\Xi_v$, or equivalently, the ranges of $P_h$ and $P_v$, 
are orthogonal complements in $L^2(\mathbb R^2)$.
\end{proof}


For later use, we denote the range spaces of $\Xi_h$ and $\Xi_v$ by
$$
  L^2_{h}(\mathbb R^2) = \Xi_h(L^2(\mathbb R^2)) \,
\quad \mbox{and} \quad
 L^2_{v} (\mathbb R^2) = \Xi_v(L^2(\mathbb R^2)) \, ,
$$
which are orthogonal complements in $L^2(\mathbb R^2)$.

Under the isometries $\Xi_v$ or $\Xi_h$, a unit norm tight frame
for $L^2(\mathcal C_v)$ or $L^2(\mathcal C_h)$ is mapped  to a unit norm
tight frame for $
  L^2_{h}(\mathbb R^2) $ or
$ L^2_{v} (\mathbb R^2)
$.
The consequence of this is that we only need to construct shearlets for the horizontal and vertical
cones, not for all of $\mathbb R^2$. If the shearlets have smoothness and
the appropriate periodicity, then we retain smoothness under the symmetrization.

\begin{corollary}
Let $g\in L^2({\mathcal C}_v)$ be a function which is continuous in
${\mathcal C}_v$ and even, then $\Xi_v g$ is continuous on $\mathbb R^2$.
If in addition there is $h: \mathbb R \to \mathbb C$ such that $g\in L^2({\mathcal C}_v)$ satisfies
$$
  g(\xi_1,\xi_2) = h(\xi_1/\xi_2, \xi_2^3/|\xi_2|)
$$
and $h$ is in $C^k(\mathbb R \times \mathbb R^*)$ and 2-periodic in its first component, then $\Xi_v g$ is $k$ times differentiable in $\mathbb R^2\setminus \{0\}$.
\end{corollary}
\begin{proof}
By the 2-periodicity, $g(\xi_1,\xi_1) = h(1,\xi_1^3/|\xi_1|)=h(-1,\xi_1^3/|\xi_1|) = g(\xi_1,-\xi_1)$. Thus, continuity of $h$
ensures that of $\Xi_v g$. A similar argument holds for differentiability.
\end{proof}

This implies that we only require a resolution of the identity on $[-1,1]$ with an appropriate
version of Gabor frames. We refer to a result of S{\o}ndergaard from \cite{Sonder}.

\begin{theorem}[\cite{Sonder}] \label{theo:sonder}
Let $N_0 \in \mathbb N$, $\alpha = 2/N_0$ and choose $\tau \in \mathbb N, \tau<N_0$.
If $w$ is a function in the Feichtinger algebra, and
if $\{M_{m\tau/2} T_{k\alpha} w \}$ is an $\frac{N_0}{\tau}$-tight Gabor frame for $L^2(\mathbb R)$,
then the periodization $w^\circ$,
$$
  w^\circ (\xi) = \sum_{n\in \mathbb Z} w(\xi - 2n) \mbox{ for a.e. } \xi \in \mathbb R \, ,
$$
defines an $\frac{N_0}{\tau}$-tight Gabor frame $\{M_{m\tau/2}T_{k\alpha}w^\circ: 0 \le k \le N_0-1, m \in \mathbb Z\}$
 for $L^2([-1,1])$.
\end{theorem}

Of particular interest to us is the following corollary, which we can draw from this result.

\begin{corollary} \label{coro:redundancy}
The uniform redundancy $\mathcal{R}^- = \mathcal{R}^+ = \frac{N_0}{\tau}$ of the
Gabor frame $\{M_{m\tau/2}T_{k\alpha}w^\circ: 0 \le k \le N_0-1, m \in \mathbb Z\}$ defined in
Theorem \ref{theo:sonder} can be chosen as close to one as desired
by choosing $N_0,\tau\in \mathbb N$ sufficiently large.

\end{corollary}

We remark that tightness is preserved when periodizing the window, and if its support is sufficiently
small then so is the norm.

Before stating the definition of cone-adapted Gabor shearlets, we require the following
additional ingredients. We consider the change of variables $(\xi_1,\xi_2) \mapsto \gamma^\iota(\xi)=
(\gamma_1^\iota(\xi),\gamma_2^\iota(\xi))$, $\iota \in \{h,v\}$, defined by
\[
     \gamma_1^h(\xi) = \frac12\sgn(\xi_1)\xi_1^2, \quad \gamma_2^h(\xi) = \frac{\xi_2}{\xi_1}
     \quad \mbox{and} \quad
     \gamma_1^v(\xi) = \frac12\sgn(\xi_2)\xi_2^2, \quad \gamma_2^v(\xi) = \frac{\xi_1}{\xi_2}.
\]
We let $\Gamma_h$ and $\Gamma_v$ denote the associated unitary operators,
$ \Gamma_h f(\xi) = f(\gamma^h(\xi))$ and $\Gamma_v f(\xi) = f(\gamma^v(\xi))$.
For each orientation $v$ or $h$, we define the appropriate dilation, shear, and modulation operators by
\[
A_j^h\equiv A_j, \quad X_m^h \equiv X_m, \quad \mbox{and } S_k^h \equiv S_k,
\]
and if $\widehat f(\xi_1,\xi_2) = \widehat g(\xi_2,\xi_1)$, then
\[
A_j^v \widehat f(\xi_1,\xi_2) = A_j^h \widehat g(\xi_2,\xi_1), \quad X_m^v \widehat f (\xi_1,\xi_2) = X_m^h \widehat g(\xi_2,\xi_1),
\quad \mbox{and }  S_k^v \widehat f(\xi_1,\xi_2) = S_k^h \widehat g(\xi_2,\xi_1).
\]


\begin{definition}
\label{defi:ConeGaborShearlets}
Let $\phi$ be an orthogonal scaling function of a $16$-band multiresolution analysis in $L^2(\mathbb R)$, with associated
orthonormal wavelets $\{\psi_\ell : \ell=1,\ldots, 15\}$, and let $N_0,\tau \in \mathbb N$ such that $w$ is the unit norm window function of an $\frac{N_0}{\tau}$-tight
Gabor frame $\{M_{m_2\tau/2}T_{2k/N_0} w: m_2, k \in \mathbb Z\}$ for $L^2(\mathbb R)$, with the periodization $w^\circ$ as described
in Theorem \ref{theo:sonder}. Let $j_0\in\bZ$. Then the associated {\em cone-adapted Gabor shearlet system} is defined by
\begin{eqnarray*}
\cC\cG\cS_{j_0}(\phi,\{\psi_\ell\}_{\ell=1}^{15}; w)
& := & \{\Phi_{j_0,k,m}^h,\Phi_{j_0,k,m}^v : k\in\bZ, |k/N_0|\le  2^{j-1}, m\in\bZ^2\}\\
& &  \cup \{\Psi_{j,k,m}^{h,\ell}, \Psi_{j,k,m}^{v,\ell} : j,k\in\bZ, j\ge j_0, |k/N_0|\le  2^{j-1}, m\in\bZ^2, \ell=1,\ldots, 15\} \subseteq L^2(\mathbb R^2),
\end{eqnarray*}
where
\[
   \widehat{\Phi_{j,k,m}^h} = \Xi_h A_j^h X^h_{(m_1,m_2\tau/2)} S^h_{2k/N_0}  \Gamma_h\widehat\phi\otimes w^\circ
   \quad \mbox{and} \quad
   \widehat{\Phi_{j,k,m}^{v}} = \Xi_v A_j^v X^v_{(m_1, m_2\tau/2)} S^v_{2k/N_0} \Gamma_v w ^\circ \otimes\widehat\phi \,
\]
and accordingly
\[
   \widehat{\Psi_{j,k,m}^{h,\ell}} = \Xi_h A_j^h X^h_{(m_1,m_2\tau/2)} S^h_{2k/N_0}  \Gamma_h\widehat\psi_\ell\otimes w^\circ
   \quad \mbox{and} \quad
   \widehat{\Psi_{j,k,m}^{v,\ell}} = \Xi_v A_j^v X^v_{(m_1, m_2\tau/2)} S^v_{2k/N_0} \Gamma_v w^\circ \otimes\widehat\psi_\ell.
\]
\end{definition}

For an illustration of the support of the special case of cone-adapted Shannon shearlets and the more general
cone-adapted Gabor shearlets, we refer to Figure \ref{fig:ConeGaborShearlet}.

\begin{figure}[h]
\begin{center}
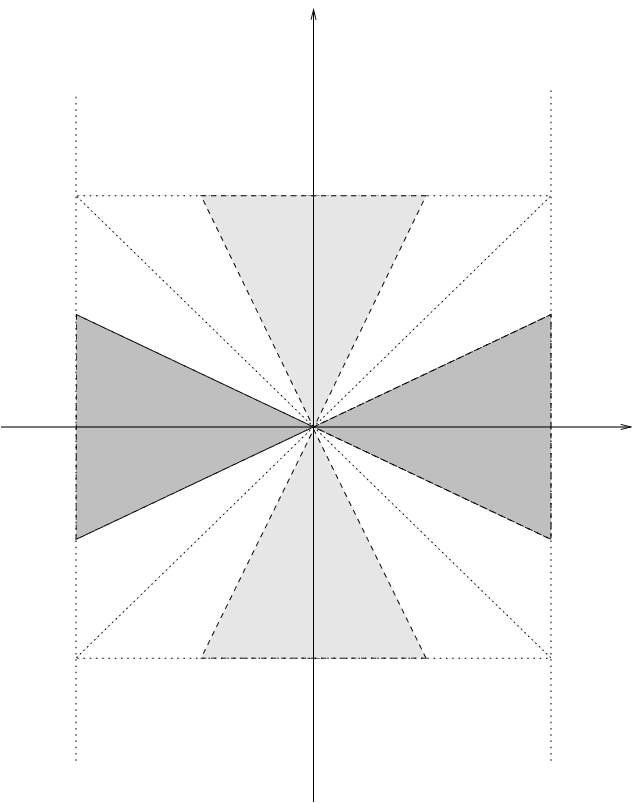
\put(-223,-10){{\footnotesize (a)}}
\hspace*{1em}
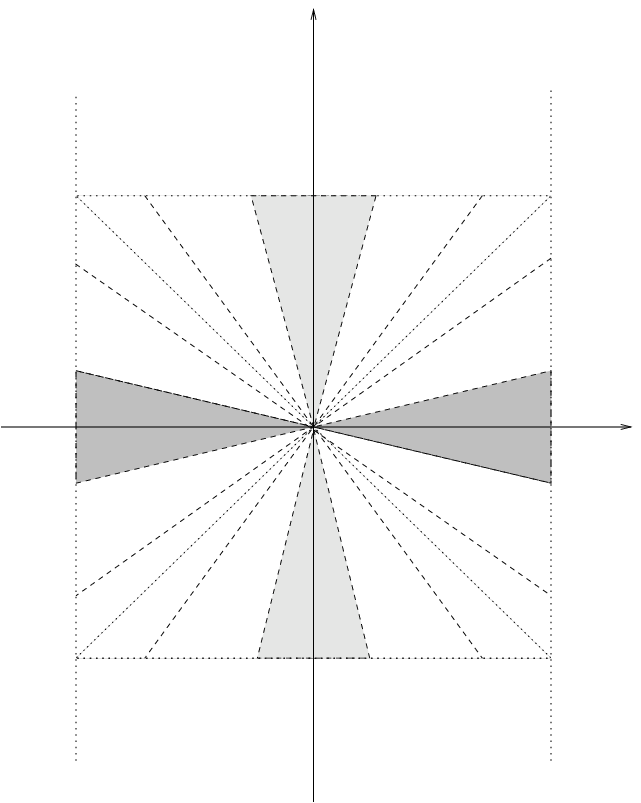
\put(-68,-10){{\footnotesize (b)}}
\end{center}
\caption{(a) Support of the cone-adapted Shannon shearlet scaling functions in the frequency domain, in horizontal and vertical 
orientations;  (b)
Support of a cone-adapted Gabor shearlet scaling function in the frequency domain, corresponding to a Gabor frame with $N_0=4$.
The smallest achievable redundancy with $N_0=4$ is obtained by setting $\tau=3$, resulting in $N_0/\tau=4/3$. With sufficiently large values of $N_0$,
and the implicit finer directional resolution,
the choice $\tau=N_0-1$ allows the redundancy to get as close to one as desired.}
\label{fig:ConeGaborShearlet}
\end{figure}


\subsection{MRA Structure}
\label{subsec:ConeMRA}

%

By classical results from frame theory, without restriction of the parameters the system consisting
of the functions $ \Phi_{j,k,m}^h$, $ \Phi_{j,k,m}^v$ forms a tight frame.

\begin{theorem}
Let $N_0 \in 2\mathbb N$, $\alpha = 2/N_0$ and $\tau \in \mathbb N$, and let
$w$ be the unit norm window function of an $\frac{N_0}{\tau}$-tight
Gabor frame $\{M_{m_2\tau/2}T_{k\alpha} w: m_2, k \in \mathbb Z\}$ for $L^2(\mathbb R)$, such that the periodization $w^\circ$
is a unit-norm Gabor frame $\{ M_{m_2\tau/2} T_{k \alpha} w^\circ \}$ for $L^2([-1,1])$, then
the system $\cC\cG\cS_{j_0}(\phi,\{\psi_\ell\}_{\ell=1}^{15}; w)$ for any $j_0\in\bZ$, or the system
\[
\{\Psi_{j,k,m}^{h,\ell}, \Psi_{j,k,m}^{v,\ell} : j,k\in\bZ, |k/N_0|\le  2^{j-1},  m\in\bZ^2, \ell=1,\ldots, 15\}
\]
 is a unit-norm $\frac{N_0}{\tau}$-tight frame for $L^2(\bR^2)$, with redundancy $\frac{N_0}{\tau}$.
\end{theorem}

\begin{proof}
The family
\[
\left\{A_j^hX^h_{(m_1,m_2\tau/2)}S_{2k/N_0}^h\Gamma_h\widehat\psi_\ell\otimes w^\circ: j,k\in\bZ, |k/N_0|\le  2^{j-1}, m\in\bZ^2; \ell=1,\ldots,
15\right\}
\]
is an $\frac{N_0}{\tau}$-tight frame for $L^2(\mathcal C_h)$.
Consequently, under the isometry,
\[
\left\{\widehat{\Psi^{h,\ell}_{j,k,m}}=
\Xi_h A_j^hX^h_{(m_1,m_2\tau/2)}S_{2k/N_0}^h\Gamma_h\widehat\psi_\ell\otimes w^\circ:
j,k\in\bZ, |k/N_0|\le  2^{j-1}, m\in\bZ^2, \ell=1,\ldots,15\right\}
\]
is an $\frac{N_0}{\tau}$-tight frame for
$P_h(L^2(\mathbb R^2))$. Similarly,
\[
\left\{\widehat{\Psi^{v,\ell}_{j,k,m}}=
\Xi_v A_j^vX^v_{(m_1,m_2\tau/2)}S_{2k/N_0}^v\Gamma_v\widehat\psi_\ell\otimes w^\circ:
j,k\in\bZ,|k/N_0|\le  2^{j-1}, m\in\bZ^2, \ell=1,\ldots,15\right\}
\]
is a $\frac{N_0}{\tau}$-tight frame for
$P_v(L^2(\mathbb R^2))$. By the orthogonality of the ranges for
$P_h$ and $P_v$, the union

\[
 \{\Psi^{h,\ell}_{j,k,m}, \Psi^{v,\ell}_{j,k,m}: j,k\in\bZ, |k/N_0|\le  2^{j-1}, m\in\bZ^2, \ell=1,\ldots,15\}
\]
is an $\frac{N_0}{\tau}$-tight frame for $L_2(\mathbb R^2)$. The proof for the case of $\cC\cG\cS_{j_0}(\phi,\{\psi_\ell\}_{\ell=1}^{15}; w)$ is similar.
\end{proof}

Corollary \ref{coro:redundancy} implies that the redundancy can be chosen arbitrarily close to one.

\begin{corollary}
The redundancy $\frac{N_0}{\tau}$ of the preceding Gabor shearlet system can be chosen
arbitrarily close to one.
\end{corollary}

\section{Optimal Sparse Approximations}
\label{sec:approx}

In this section, we show that under certain assumptions the cone-adapted Gabor shearlet system $\cC\cG\cS_{j_0}(\phi,\{\psi_\ell\}_{\ell=1}^{15};w)$
as defined in Definition \ref{defi:ConeGaborShearlets} provides optimally sparse approximation of cartoon-like functions,
similar to `classical' shearlets (see \cite{GL07a,KL10}). Due to the asymptotic nature of the optimally approximation results, which involve only with shearlets with large scale $j$, without loss of generality, we consider $\cC\cG\cS_{j_0}(\phi,\{\psi_\ell\}_{\ell=1}^{15};w)$ with $j_0=0$ and denote the system as $\cC\cG\cS(\Psi^h, \Psi^v)$. We will first state the main result and the core proof in the
following subsection, and postpone the very technical parts of the proof to later subsections.

\subsection{Main Result}
\label{subsec:main}


We first require the definition of cartoon-like functions. For this, we recall that in \cite{CD04} $\mathcal{E}^2(A)$
denotes the set of {\em cartoon-like functions} $f$, which are $C^2$ functions  away from a $C^2$ edge singularity:
$f=f_0+f_1\chi_B$, where $f_0,f_1\in C^2([0,1]^2)$ and $\|f\|_{C^2}:=\sum_{|v|\le2}\|\partial^vf\|_\infty\le 1$ with
$\partial^v=\partial_1^{v_1}\partial_2^{v_2}$ being the 2D differential operator with order $\partial_1=\frac{\partial}{\partial x_1}$,
$\partial_2=\frac{\partial}{\partial x_2}$, and $v=(v_1,v_2)$. More precisely, in polar coordinates, let
$\rho(\theta):[0,2\pi)\mapsto[0,1]^2$ be a radius function satisfying $\sup_\theta|\rho''(\theta)|\le A$ and $\rho\le\rho_0\le
1$. The set $B\subset\bR^2$ is given by $B = \{x \in [0,1]^2: ||x||_2\le \rho(\theta)\}$. In particular, the boundary $\partial B$
is given by the curve in $\bR^2$: $ \beta(\theta)=(\rho(\theta)\cos\theta, \rho(\theta)\sin\theta)$.

Utilizing this notion, we can now formulate out main result concerning optimal sparse approximation of such cartoon-like functions
by our cone-adapted Gabor shearlet system as follows.

\begin{theorem}\label{thm:main2}
Let $f\in \mathcal{E}^2(A)$ and $f_N$ be the $N$-term approximation
of $f$ from the $N$  largest cone-adapted Gabor shearlet coefficients
$\{\ip{f}{\Psi_\mu}:\Psi_\mu\in\cC\cG\cS(\Psi^h,\Psi^v)\}$ in magnitude. Then
\[
\|f-f_N\|_2^2\le c\cdot N^{-2}\cdot (\log N)^3.
\]
\end{theorem}

To prove this theorem, we  follow the main idea as in
\cite{CD04,GL07a}. In a nutshell, we first use a smooth partition of unity that decomposes a cartoon-like function $f$ into
small dyadic cubes of size about $2^{-j}\times 2^{-j}$. If  $j$ is large enough, then there are only  two types of dyadic cubes: one intersects with the
singularity of the function, namely, the edge fragements,  and the other only contains the smooth region of the function. We then analyze the decay property of the shearlet coefficients. Eventually, by combining the decay estimation of each dyadic cube, we can prove Theorem~\ref{thm:main2}.

Though the main steps are similar to \cite{CD04,GL07a}, we however would like to point out that
some of the key steps require slightly technical extensions of results in  \cite{CD04, GL07a}.
For the results available in \cite{CD04, GL07a}, we simply state them here without proof for the purpose of readability.

Let us next state some necessary auxiliary results, including Theorem~\ref{thm:coeffEdge} for the decay estimate with respect to those edge fragements and Theorem~\ref{thm:coeffSmooth}  for the decay estimate with respect to those smooth regions.

\begin{figure}[h]
\begin{center}
\includegraphics[width=3.0in]{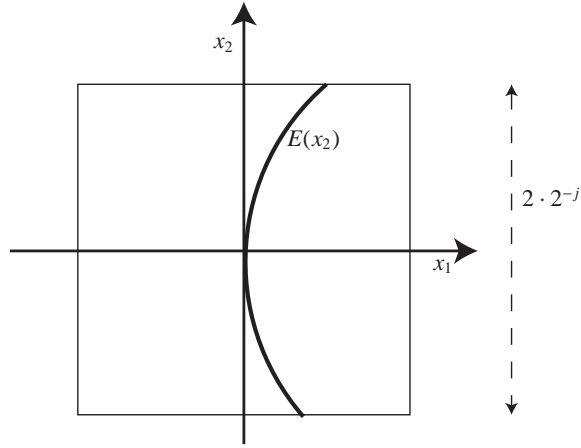}
\put(-140,150){\small {$x_2$}}
\put(-57,66){\small {$x_1$}}
\put(-112,113){\small {$E(x_2)$}}
\put(-23,90){\small {$2 \cdot 2^{-j}$}}
\end{center}
\caption{An edge fragment.}
\label{fig:edgefragment}
\end{figure}
An edge fragement (see Figure~\ref{fig:edgefragment}) is of the form
\begin{equation}\label{def:edgeFrag}
f(x_1,x_2) = w_0(2^{j}x_1,2^{j}x_2)g(x_1,x_2){\bf1}_{\{x_1\ge E(x_2)\}},
\end{equation}
where $w_0,g$ are smooth functions supported on $[-1,1]^2$ and
$|E''(x)|\le A$.

Let $\mathcal{Q}_j$ be the collection of dyadic cubes of the form $Q=[m_1/2^j,(m_1+1)/2^j]\times[m_2/2^j,(m_2+1)/2^j]$. For $w_0$ a nonnegative $C^\infty$ function with support in $[-1,1]^2$, we can define a smooth partition of unity
\[
\sum_{Q\in\mathcal{Q}_j}w_Q(x)=1,  \quad x\in\bR^2
\]
with $w_Q=w_0(2^jx_1-m_1,2^jx_2-m_2)$.
If $Q\in\mathcal{Q}_j$ intersects with the curve singularity, then $f_Q:=fw_Q$ is an edge fragment.

Let $\mathcal{Q}_j^0$ be the collection of those dyadic cubes $Q\in\mathcal{Q}_j$ such that the edge singularity intersects with the support of $w_Q$. Then the cardinality
\begin{equation}
|\mathcal{Q}_j^0|\le c\cdot 2^{j}.
\end{equation}
Similarly, $\mathcal{Q}_j^1:=\mathcal{Q}_j\backslash\mathcal{Q}_j^0$ are those cubes that do not intersect with the edge singularity. We have
\begin{equation}
|\mathcal{Q}_j^1|\le c\cdot 2^{2j}+4\cdot 2^j.
\end{equation}

Let $\{s_\mu\}$ be a sequence. We define $|s_\mu|_{(N)}$ to be the $N$th largest entry of the $\{|s_\mu|\}$. The \emph{weak-$\ell^p$ quasi-norm $\|\cdot\|_{w\ell^p}$} of $\{s_\mu\}$ is defined to be
\[
\|s_\mu\|_{w\ell^p}:=\sup_{N>0} \left(N^{1/p}\cdot|s_\mu|_{(N)}\right),
\]
which is equivalent to
\[
\|s_\mu\|_{w\ell^p}=\left(\sup_{\epsilon>0} (\#\{\mu: |s_\mu|>\epsilon\}\cdot \epsilon^p)\right)^{1/p},
\]
We abbreviate indices for elements in $\cC\cG\cS(\Psi^h,\Psi^v)$ and write $\Psi_\mu$ with $\mu=(j,k,m;\iota,\ell)$. The index set at scale $j$ is
$\Lambda_j:=\{\mu=(j,k,m;\iota,\ell): k\in\bZ, |k/N_0|\le 2^{j-1}, m\in\bZ^2;
\ell=1,\ldots,15, \iota=h,v\}$.

Now similar to \cite[Theorem 1.3]{GL07a}, we have the following result which provides a decay estimate of
the coefficients with respect to those $Q\in\mathcal{Q}_j^0$.
\begin{theorem}
\label{thm:coeffEdge}
Let $f\in \mathcal{E}^2(A)$ and $f_Q:=fw_Q$. For $Q\in \mathcal{Q}_j^0$ with $j\ge0$ fixed, the sequence of coefficients $\{\ip{f_Q}{\Psi_\mu}: \mu\in \Lambda_j\}$ obeys
\[
\|\ip{f_Q}{\Psi_\mu}\|_{w\ell^{2/3}}\le c\cdot 2^{-3j/2}
\]
for some constant c independent of $Q$ and $j$.
\end{theorem}

Similarly, for the smooth part, we can show
that the sequence of coefficients $\{\ip{f_Q}{\Psi_\mu}: \mu\in
\Lambda_j\}$ with $Q\in\mathcal{Q}_j^1$ obeys the following estimate (c.f. \cite[Theorem~1.4]{GL07a}).
\begin{theorem}
\label{thm:coeffSmooth}
Let $f\in \mathcal{E}^2(A)$. For $Q\in \mathcal{Q}_j^1$ with $j\ge0$ fixed, the sequence of coefficients $\{\ip{f_Q}{\Psi_\mu}: \mu\in \Lambda_j\}$ obeys
\[
\|\ip{f_Q}{\Psi_\mu}\|_{w\ell^{2/3}}\le c\cdot 2^{-3j}
\]
for some constant independent of $Q$ and $j$.
\end{theorem}

The proofs of Theorems~\ref{thm:coeffEdge} and~\ref{thm:coeffSmooth} are very technical and require extension of
results in \cite{CD04,GL07a}. We therefore postpone their detailed proofs to the next two subsections.
As a consequence of Theorem~\ref{thm:coeffEdge} and Theorem~\ref{thm:coeffSmooth},  it is easy to show the following result.
\begin{corollary}\label{cor:coeffwl}
Let $f\in\mathcal{E}^2(A)$ and for $j\ge 0$, let $s_j(f)$ be the sequence of $s_j(f)=\{\ip{f}{\Psi_\mu}: \mu\in \Lambda_j\}$. Then
\[
\|s_j(f)\|_{w\ell^{2/3}}\le c
\]
\end{corollary}
\begin{proof}By the triangle inequality,
\[
\begin{aligned}
\|s_j(f)\|_{w\ell^{2/3}}^{2/3}&\le \sum_{Q\in\mathcal{Q}_j}\|\ip{f_Q}{\Psi_\mu}\|^{2/3}_{w\ell^{2/3}}
\\&\le\sum_{Q\in\mathcal{Q}_j^0}\|\ip{f_Q}{\Psi_\mu}\|^{2/3}_{w\ell^{2/3}}+\sum_{Q\in\mathcal{Q}_j^1}\|\ip{f_Q}{\Psi_\mu}\|^{2/3}_{w\ell^{2/3}}
\\&\le c\cdot |\mathcal{Q}_j^0|\cdot 2^{-j}+c\cdot |\mathcal{Q}_j^1|\cdot 2^{-2j}
\\&\le c.
\end{aligned}
\]
\end{proof}

Now, we can give the decay rate of our cone-adapted Gabor shearlet coefficients as follows.
\begin{theorem}\label{thm:main1}
Let $f\in \mathcal{E}^2(A)$ and $s(f):=\{\ip{f}{\Psi_\mu}:
\Psi_\mu\in\cC\cG\cS(\Psi^h,\Psi^v)\}$ be the   cone-adapted Gabor shearlet
coefficients associated with $f$. Let $\{|s(f)|_{(N)} :
N=1,2,\ldots\}$ be the sorted sequence of the absolute values of
$s(f)$ in descending order. Then
\[
\sup_{f\in\mathcal{E}^2(A)}|s(f)|_{(N)}\le c\cdot N^{-3/2}\cdot (\log N)^{3/2}.
\]
\end{theorem}
\begin{proof}
From Definition~\ref{defi:ConeGaborShearlets}, we have
$\widehat{\Psi^{h,\ell}_{j,k,m}}=\Xi_h g^h_{j,k,m}$ with $g^{h,\ell}_{j,k,m}=A_j^h X^h_{(m_1,m_2\tau/2)}S^h_{2k/N_0}\Gamma_h \widehat{\psi_\ell}\otimes w^\circ$. Then,

\[
\begin{aligned}
\widehat{\Psi^{h,\ell}_{j,k,m}}(\xi_1,\xi_2)&=\overline{H_+}(I+\frac{1+i}{2}R+\frac{1-i}{2}R^3)g^{h,\ell}_{j,k,m}(\xi_1,\xi_2)\\
&=
\overline{H(\zeta(\xi_2/\xi_1))}g^{h,\ell}_{j,k,m}(\xi_1,\xi_2)\\&\quad+
\overline{H(\zeta(\xi_2/\xi_1))}(\frac{1+i}{2}g^{h,\ell}_{j,k,m}(\xi_2,-\xi_1)+\frac{1-i}{2}g^{h,\ell}_{j,k,m}(-\xi_2,\xi_1))
\\&=:g_1+g_2.
\end{aligned}
\]
For analyzing the optimal
sparsity, we first consider $\Theta^{h,\ell}_{j,k,m} =g_1$, which
can be rewritten as follows:
\[
\begin{aligned}
\Theta_{j,k,m}^{h,\ell}(\xi_1,\xi_2)&\equiv\sigma_{j,k}^\ell(\gamma^h(\xi))\cdot
e_{j,m}(\gamma^h(\xi))
\end{aligned}
\]
with
\begin{equation}
\sigma_{j,k}^\ell(\gamma^h(\xi)):=\overline{H(\zeta(\gamma^h(\xi)))}\;\widehat\psi_\ell(2^{-4j}\gamma_1^h(\xi))w^\circ(2^{j}\gamma_2^h(\xi)-\frac{2k}{N_0})
\end{equation}
and
\begin{equation}
e_{j,m}(\gamma^h(\xi)):=2^{-3j/2}e^{2\pi i
m_12^{-4j}\gamma_1^h(\xi)}  e^{2\pi i  \frac{m_2\tau}{2}2^{j}\gamma_2^h(\xi)}.
\end{equation}
 For simplicity,
we use again the compact notation
$\Theta_{\mu}(\xi):=\sigma_{j,k}^\ell(\gamma^\iota(\xi))e_{j,m}(\gamma^\iota(\xi))$
with $\mu=(j,k,m;\iota,\ell)\in\Lambda_j$. The index set $\Lambda_j$ at scale $j$ is as before.

By Corollary~\ref{cor:coeffwl}, we have
\[
R(j,\epsilon):=\#\{\mu\in \Lambda_j: |\ip{f}{\Theta_\mu}|>\epsilon\}\le c\cdot \epsilon^{-2/3}.
\]
Also,
\[
|\ip{f}{\Theta_\mu}|\le c\cdot 2^{-3j/2}.
\]
Therefore, $R(j,\epsilon)=0$ for $j>\frac23\log_2(\epsilon^{-1})$. Thus
\[
\#\{\mu: |\ip{f}{\Psi_\mu}|>\epsilon\} \le \sum_{j\ge 0} R(j,\epsilon)\le c\cdot \epsilon^{-2/3}\cdot \log_2(\epsilon^{-1}),
\]

Repeating the steps for the second term in the definition of $\Psi^{h,\ell}_{j,k,m}$ shows that we can replace $\Theta_\mu$
by $\Psi_\mu$ at the cost of  a change of the constant $c$.
This can be seen from the fact that the term $\frac{1+i}{2}g^{h,\ell}_{j,k,m}(\xi_2,-\xi_1)+\frac{1-i}{2}g^{h,\ell}_{j,k,m}(-\xi_2,\xi_1)$
 is supported in the vertical cone and thus $g_2$ can be viewed as composed of two
quarter-rotated elements of the form of $g_1$.
The same strategy applies to the vertical cone elements.
The theorem is proved.
\end{proof}

Now we can prove Theorem~\ref{thm:main2} using the above results.
\begin{proof}[Proof of Theorem~\ref{thm:main2}]
$f_N=\sum_{\mu\in I_N} \ip{f}{\Psi_\mu}\Psi_\mu$ where $I_N$ is the
set of indices corresponding to the $N$ largest entry of
$\{|\ip{f}{\Psi_\mu}|: \mu\}$. By the tight frame property and Theorem~\ref{thm:main1}, we have
\[
\|f-f_N\|^2\le \sum_{n>N}|s(f)|_{(N)}^2 \le c\cdot \sum_{n>N}  N^{-3}\log(N)^3\le c\cdot N^{-2}\cdot \log(N)^3.
\]
This finishes the proof of the theorem.
\end{proof}

\subsection{Analysis of the Edge Fragments}
\label{subsec:edgefragments}

We shall focus on proving Theorem~\ref{thm:coeffEdge} next. To that end, we need some auxiliary results first.
From \cite[Theorem 2.2]{GL07a} or  \cite[Theorem 6.1]{CD04}, we have the following result, which gives the estimate of the decay of the edge fragment in the Fourier domain along a fixed direction.
\begin{theorem}\label{thm:edgeEstimate}
Let $f$ be an edge fragment as defined in \eqref{def:edgeFrag} and $I_j:=[2^{2j-\alpha},2^{2j+\beta}]$ with $\alpha\in\{0,1,2,3,4\}$ and $\beta\in\{0,1,2\}$.
Then,
\[
\int_{|\lambda|\in I_j}|\widehat f(\lambda \cos\theta,\lambda\sin\theta)|^2d\lambda\le c\cdot 2^{-4j}\cdot (1+2^{j}|\sin\theta|)^{-5}.
\]
\end{theorem}

Use Theorem~\ref{thm:edgeEstimate}, one can prove  the following result (c.f. \cite[Proposition 2.1]{GL07a}).
\begin{corollary}\label{cor:edgeEstimate2}
Let $f$ be an edge fragment as defined in \eqref{def:edgeFrag}. Then
\[
\int_{\bR^2}|\widehat f(\xi)|^2|\sigma_{j,k}^\ell(\gamma^h(\xi))|^2d\xi\le
c\cdot2^{-3j}(1+|k|)^{-5}.
\]
\end{corollary}
Note that although  $\sigma_{j,k}^\ell(\gamma^h(\xi))$  might not be compactly
supported compared to \cite[Proposition 2.1]{GL07a}, it does not affect the result here, since proofs related
to the support of $\sigma_{j,k}^\ell(\gamma^h(\xi))$ can be passed through its essential
support and the estimate outside the essential supported is absorbed
in the constant $c$. For elements in the vertical cone
$\sigma_{j,k}^\ell(\gamma^v(\xi))$, similarly to the above result,
one can show that the decay estimate is of order less than
$2^{-3j}(1+|k|)^{-5}$.

From \cite[Corollary 2.4]{GL07a} or  \cite[Corollary 6.6]{CD04}, we have the following result about the decay of the derivative of the edge fragment in the Fourier domain along a fixed direction.
\begin{corollary}\label{cor:edgeEstimate}
Let $f$ be an edge fragment as defined in \eqref{def:edgeFrag} and
$v=(v_1,v_2)$.Then
\[
\int_{|\lambda|\in I_j}|\partial^v\widehat f(\lambda \cos\theta,\lambda\sin\theta)|^2d\lambda\le c_v\cdot 2^{-2j|v|}\cdot 2^{-2jv_1}\cdot 2^{-4j}\cdot (1+2^{j}|\sin\theta|)^{-5}+c_v\cdot 2^{-2j|v|}\cdot 2^{-10j}.
\]
\end{corollary}
We also need the following lemma (see \cite[Lemma 2.5]{GL07a}),
which follows from a direct computation.
\begin{lemma}\label{lem:supSigma}
Let $\sigma_{j,k}^\ell(\gamma^h(\xi))$ be given as above. Then, for
each $v=(v_1,v_2)\in\bN^2$, $v_1,v_2\in\{0,1,2\}$,
\[
\left|\partial^v\sigma_{j,k}^\ell(\gamma^h(\xi))\right|\le c_v\cdot
2^{-(2v_1+v_2)j}\cdot (1+|k|)^{v_1},
\]
where $|v|=v_1+v_2$ and $c_v$ is independent of $j$ and $k$.
\end{lemma}

Use the above results, we  can prove the following result, which is an
extension of \cite[Proposition~2.3]{GL07a} and can be proved with similar approach.
\begin{corollary}\label{cor:edgeEstimate3}
Let $f$ be an edge fragment defined as in \eqref{def:edgeFrag},
$\sigma_{j,k}^\ell(\gamma^h(\xi))$ be defined as above, and $L_t$ be
the differential operator defined by
\[
 L_t=\left(t\cdot I-\left(\frac{2^{2j}}{2\pi(1+|k|)}\right)^2\partial_1^2\right)\left(I-\left(\frac{2^j}{2\pi}\right)^2\partial_2^2\right),
\]
where $t>0$ is a fixed constant. 
Then
\[
\int_{\bR^2}\left|L_t\left(\widehat
f(\xi)\sigma_{j,k}^\ell(\gamma^h(\xi))\right)\right|^2d\lambda\le
c_t\cdot2^{-3j}(1+|k|)^{-5}
\]
for some positive constant $c_t$ independent of $j$ and $k$.
\end{corollary}

Now we are ready to prove Theorem~\ref{thm:coeffEdge}.
\begin{proof}[Proof of Theorem~\ref{thm:coeffEdge}]
Fix $j\ge0$, for simplicity, let $f=f_Q$ be the edge fragment as in \eqref{def:edgeFrag}. We have
\[
\ip{f}{\Psi_\mu}=\int_{\bR^2}\widehat
f(\xi)\overline{\sigma_{j,k}^\ell(\gamma^h(\xi))}\cdot\overline{e_{j,m}(\gamma^h(\xi))}d\xi.
\]
We have
\[
\partial_1{e_{j,m}(\gamma^h(\xi))}=(2\pi im_12^{-4j}\sgn(\xi_1)\xi_1-2\pi i \frac{m_2\tau}{2} 2^j\frac{\xi_2}{\xi_1^2}){e_{j,m}(\gamma^h(\xi))}.
\]

\[
\begin{aligned}
\partial_1^2{e_{j,m}(\gamma^h(\xi))}&=(2\pi im_12^{-4j}\sgn(\xi_1)+2\pi i  \frac{m_2\tau}{2} 2^{j+1}\frac{\xi_2}{\xi_1^3}){e_{j,m}(\gamma^h(\xi))}\\&
+(2\pi im_12^{-4j}\sgn(\xi_1)\xi_1-2\pi i  \frac{m_2\tau}{2} 2^{j}\frac{\xi_2}{\xi_1^2})^2{e_{j,m}(\gamma^h(\xi))}.
\end{aligned}
\]
Also,
\[
\partial_2^2{e_{j,m}(\gamma^h(\xi))}=(2\pi i  \frac{m_2\tau}{2} \frac{2^{j}}{\xi_1})^2{e_{j,m}(\gamma^h(\xi))}.
\]
Let $L_t$ be the differential operator defined in Corollary~\ref{cor:edgeEstimate3}. Then,
\[
L_t({e_{j,m}(\gamma^h(\xi))}) = g_{j,k}^m(\xi) {e_{j,m}(\gamma^h(\xi))}
\]
with
\[
\begin{aligned}
g_{j,k}^m(\xi) &=\left[t+\frac{(\frac{m_1}{2\pi i}\sgn(\xi_1)+\frac{ 1 }{2\pi i}\frac{m_2\tau}{2}\frac{2^{5j+1}\xi_2}{\xi_1^3})
+(m_1\frac{\sgn(\xi_1)\xi_1}{2^{2j}}- \frac{m_2\tau}{2} \frac{2^{3j}\xi_2}{\xi_1^2})^2}{(1+|k|)^2}\right]\cdot\left[1+ (\frac{m_2\tau}{2}\frac{2^{2j}}{\xi_1})^2\right]
\end{aligned}
\]
Let $W_{j,k}$ be the essential support of
$\sigma_{j,k}^\ell(\gamma^h(\xi))$ defined as
\begin{equation}\label{suppWjk}
W_{j,k}:=\left\{ (\lambda,\theta): 2^{2j}a'\le|\lambda|\le 2^{2j}b',  \arctan(2^{-j}(\frac{2k}{N_0}-1))\le \theta\le \arctan(2^{-j}(\frac{2k}{N_0}+1))\right\}.
\end{equation}
For $\xi\in W_{j,k}$, we have
$|\xi_1|\approx 2^{2j}$ and one can show that
$\frac{2^j\xi_2}{\xi_1}\approx 2^j\tan\theta\approx k$.
Consequently, we can choose large $t>0$ independent of $j,k, m$ such
that
\begin{equation}
 \sup_{\xi\in W_{j,k}}|g_{j,k}^m(\xi)|\ge c \cdot  \left[1+\frac{(  m_1-  m_2 k)^2}
{(1+|k|)^2}\right]\cdot\left[1+m_2^2
\right]=:c\cdot G_{k}(m)
\end{equation}
for some positive constant $c$ independent of $j$, $k$, and $m$.  For $G_k(m)$, we have
\[
\begin{aligned}
G_{k}(m)= \begin{cases}
 (1+m_1^2)(1+m_2^2) & \mbox{for } k=0\\
 \left[1+\frac{( \frac{m_1}{k}-  m_2)^2}{(1+|k|)^2/|k|^2}\right]\cdot\left[1+m_2^2
\right]  &\mbox{for } k\neq0\\
\end{cases}
\end{aligned}
\]
Consequently,
\[
\begin{aligned}
\ip{f}{\Psi_\mu}&=\int_{\bR^2}\widehat
f(\xi)\overline{\sigma_{j,k}^\ell(\gamma^h(\xi))}\cdot\overline{e_{j,m}(\gamma^h(\xi))}d\xi
\\&=\int_{\bR^2}L_t(\widehat f(\xi)\overline{\sigma_{j,k}^\ell(\gamma^h(\xi))})\cdot L_t^{-1}(\overline{e_{j,m}(\gamma^h(\xi))})d\xi
\\&=\int_{\bR^2}\frac{L_t(\widehat f(\xi)\overline{\sigma_{j,k}^\ell(\gamma^h(\xi))})}{ \overline{g_{j,k}^m(\xi)}}\cdot\overline{e_{j,m}(\gamma^h(\xi))}d\xi.
\end{aligned}
\]

For $k\neq 0$  and $\bm:=(\bm_1,\bm_2)\in\bZ^2$, define $R_{\bm}:=\{m=(m_1,m_2)\in\bZ^2: \frac{m_1}{k}\in[\bm_1,\bm_1+1), m_2=\bm_2\}$.
Since for $j,k$ fixed, $\{e_{j,m}(\gamma^h(\xi)): k\in\bZ^2\}$ is an orthonormal basis for $L^2$ functions supported on $W_{j,k}$, we obtain
\[
\begin{aligned}
\sum_{m\in R_{\bm}}|\ip{f}{\Psi_\mu}|^2&\le
\int_{\bR^2}\left|\frac{L_t(\widehat
f(\xi)\overline{\sigma_{j,k}^\ell(\gamma^h(\xi))})}{
\overline{g_{j,k}^m(\xi)}}\right|^2d\xi
\\&\le\sup_{\xi \in W_{j,k}} \frac{1}{|g_{j,k}^m(\xi)|^2}
\int_{\bR^2}\left|L_t(\widehat
f(\xi)\overline{\sigma_{j,k}^\ell(\gamma^h(\xi))})\right|^2d\xi
\\&\le c\cdot  \frac{1}{[(1+(m_1/k-m_2)^2)(1+m_2^2)]^2}
\int_{\bR^2}\left|L_t(\widehat
f(\xi)\overline{\sigma_{j,k}^\ell(\gamma^h(\xi))})\right|^2d\xi.
\end{aligned}
\]
By Corollary~\ref{cor:edgeEstimate3}, we have
\[
\begin{aligned}
\sum_{m\in R_{\bm}}|\ip{f}{\Psi_\mu}|^2&
\le c\cdot G_\bm^{-2}\cdot
2^{-3j}(1+|k|)^{-5}
\end{aligned}
\]
with $G_\bm:= (1+(\bm_1-\bm_2)^2)(1+\bm_2^2)$. For $k=0$, similarly, we have $G_\bm=(1+\bm_1^2)(1+\bm_2^2)$.

Let $N_{j,k,\bm}(\epsilon):=\#\{m\in R_\bm: |\ip{f}{\Psi_\mu}|>\epsilon\}$. Then
$N_{j,k,\bm}(\epsilon)\le c\cdot(1+|k|)$ and the above inequality implies
\[
N_{j,k,\bm}(\epsilon)\le  c\cdot G_\bm^{-2}\cdot 2^{-3j}\cdot \epsilon^{-2} \cdot (1+|k|)^{-5}.
\]
Thus,
\[
N_{j,k,\bm}(\epsilon)\le c\cdot\min(1+|k|, G_\bm^{-2}\cdot 2^{-3j}\cdot \epsilon^{-2} \cdot (1+|k|)^{-5}),
\]
which implies
\[
\sum_{k=-2^j}^{2^j}N_{j,k,\bm}(\epsilon)\le c\cdot
G_\bm^{-2/3}\cdot 2^{-j}\cdot \epsilon^{-2/3}.
\]
Since $\sum_{\bm\in\bZ^2} G_\bm^{-2/3}<\infty$, by above inequality, we obtain
\[
\#\{\mu\in M_j: |\ip{f}{\Psi_\mu}|>\epsilon\}\le
\sum_{m\in\bZ^2}\sum_{k=- 2^j}^{2^j} N_{j,k,\bm}(\epsilon)\le
c\cdot  2^{-j} \epsilon^{-2/3},
\]
which is equivalent to the conclusion that
\[
\|\ip{f_Q}{\Psi_\mu}\|_{w\ell^{2/3}}\le c\cdot2^{-3j/2}.
\]
\end{proof}

\subsection{Analysis of the Smooth Region}
\label{subsec:smoothregion}

Now, we shall focus on proving Theorem~\ref{thm:coeffSmooth}.
Let us provide some lemmas first.
From
\cite[Lemma 8.1]{CD04} or \cite[Lemma 2.6]{GL07a}, we have
\begin{lemma}\label{lem:smooth1}
Let $f=gw_Q$, where $g\in \mathcal{E}^2(A)$ and $Q\in\mathcal{Q}_j^1$. Then
\[
\int_{W_{j,k}}|\widehat f(\xi)|^2d\xi\le c\cdot 2^{-10j},
\]
where
$W_{j,k}$ is the essential support
of $\sigma_{j,k}^\ell(\gamma^h(\xi))$ as in \eqref{suppWjk}.
\end{lemma}

From \cite[Lemma 2.7]{GL07a} we have
\begin{lemma}\label{lem:smooth2} for $v=(v_1,v_2)\in \bN^2$,
\[
\sum_{k=-2^j}^{2^j}\left|\partial^v\sigma_{j,k}^\ell(\gamma^h(\xi))\right|^2\le
c\cdot  2^{-2|v|j}.
\]
\end{lemma}

Using the above two lemmas, one can easily prove the following result, which is an extension of \cite[Lemma 2.8]{GL07a} and can be proved by a similar approach.
\begin{lemma}\label{lem:smooth3} Let $f=gw_Q$, where $g\in\mathcal{E}^2(A)$ and $Q\in\mathcal{Q}_j^1$.
Define the differential operator
$L_t:=(tI-\frac{2^{2j}}{(2\pi)^2}\Delta)$ with $t>0$ and
$\Delta=\partial_1^2+\partial_2^2$. Then,
\[
\int_{\bR^2} \sum_{k=-2^j}^{2^{j}}\left|L_t^2(\widehat f(\xi)
\sigma_{j,k}^\ell(\gamma^h(\xi)))\right|^2d\xi\le c_t\cdot 2^{-10j}
\]
for some positive constant $c_t$ independent of $j$.
\end{lemma}

Now we are ready to prove Theorem~\ref{thm:coeffSmooth}.
\begin{proof}[Proof of Theorem~\ref{thm:coeffSmooth}]
Let $f=f_Q=gw_Q$ and $L_t$ as defined in Lemma~\ref{lem:smooth3}. We
have
\[
L_t(e_{j,m}(\gamma^h(\xi))= g_{j,k}^m(\xi) e_{j,m}(\gamma^h(\xi)),
\]
where
\[
\begin{aligned}
g_{j,k}^m(\xi)&=\Big[t+\frac{m_12^{-2j}\sgn(\xi_1)+ \frac{m_2\tau}{2} 2^{3j+1}\frac{\xi_2}{\xi_1^3}}{2\pi i}+2^{2j}\Big(m_12^{-4j}\sgn(\xi_1)\xi_1-\frac{m_2\tau}{2} 2^{j}\frac{\xi_2}{\xi_1^2}\Big)^2
\\&+(\frac{m_2\tau}{2}\frac{2^{2j}}{\xi_1})^2\Big]e_{j,m}(\gamma^h(\xi))
\end{aligned}
\]
Similar argument to the proof of Theorem~\ref{thm:coeffEdge}, we can choose $t>0$ large enough so that
\[
 \sup_{\xi\in W_{j,k}}|g_{j,k}^m(\xi)|\ge c \cdot[1+2^{-2j}(m_1-m_2k)^2+m_2^2)].
\]
For $\bm:=(\bm_1,\bm_2)\in\bZ^2$, define $R_{\bm}:=\{m=(m_1,m_2)\in\bZ^2: 2^{-2j}(m_1-m_2 k)\in[\bm_1,\bm_1+1), m_2=\bm_2\}$. Observe that for each $\bm$, there are only $1+2^{2j}$ choices for $m_1$ in $R_{\bm}$. Hence $\# R_\bm\le 1+2^{2j}$. Again, similar argument to the proof of Theorem~\ref{thm:coeffEdge}, we have
\[
\begin{aligned}
\sum_{m\in R_{\bm}}|\ip{f}{\Psi_\mu}|^2& \le c\cdot\sup_{\xi \in
W_{j,k}} \frac{1}{|g_{j,k}^m(\xi)|^4} \int_{\bR^2}\left|L_t^2(\widehat
f(\xi)\overline{\sigma_{j,k}^\ell(\gamma^h(\xi))})\right|^2d\xi
\\&\le c\cdot
\frac{1}{[1+2^{-2j}(m_1-m_2k)^2+m_2^2)]^4}
\int_{\bR^2}\left|L_t^2(\widehat
f(\xi)\overline{\sigma_{j,k}^\ell(\gamma^h(\xi))})\right|^2d\xi.
\end{aligned}
\]
Then by Lemma~\ref{lem:smooth3},
\[
\begin{aligned}
\sum_{k=-2^j}^{2^j}\sum_{m\in
R_{\bm}}|\ip{f}{\Psi_\mu}|^2& \le c\cdot G_\bm^{-4}\cdot
\int_{\bR^2}\sum_{k=-2^j}^{2^j}\left|L_t^2(\widehat
f(\xi)\overline{\sigma_{j,k}^\ell(\gamma^h(\xi))})\right|^2d\xi
\\&\le c \cdot G_\bm^{-4} \cdot 2^{-10j}.
\end{aligned}
\]
where $G_\bm:=1+\bm_1^2+\bm_2^2$.

Using the H\"older inequality
\[
\sum_{m=1}^N|a_m|^p\le \left(\sum_{m=1}^N|a_m|^2\right)^{p/2}N^{1-p/2},\quad 1/2<p<2.
\]
Since  the cardinality of $R_\bm$ is bounded by $1+2^{2j}$, we have
\[
\sum_{k=-2^j}^{2^j}\sum_{m\in R_{\bm}}|\ip{f}{\Psi_\mu}|^p\le
c\cdot (2^{2j})^{1-p/2}\cdot G_\bm^{-2p}\cdot  2^{-5pj}
\]
Moreover, since $p>1/2$, $\sum_{\bm\in\bZ^2} G_\bm^{-2p}<\infty$. Consequently,
\[
\sum_{\mu\in M_j}|\ip{f}{\Psi_\mu}|^p\le c\cdot
2^{2j(1-p/2)-5pj}=c\cdot   2^{2j(1-3p)}.
\]
In particular
\[
\|\ip{f}{\Psi_\mu}\|_{l^{2/3}}\le c\cdot  2^{-3j}.
\]
\end{proof}


\bigskip

\end{document}